\documentclass[12pt]{amsart}
\usepackage{amssymb,amsmath,amsthm,hyperref}
\usepackage[mathscr]{eucal}
\usepackage[small]{diagrams}

\usepackage{enumerate}

\newcommand{\abar}{{\ensuremath{\bar{a}}}}
\newcommand{\bbar}{{\ensuremath{\bar{b}}}}
\newcommand{\cbar}{{\ensuremath{\bar{c}}}}

\newcommand{\xbar}{{\ensuremath{\bar{x}}}}
\newcommand{\ybar}{{\ensuremath{\bar{y}}}}
\newcommand{\zbar}{{\ensuremath{\bar{z}}}}

\newcommand{\Xbar}{{\ensuremath{\bar{X}}}}
\newcommand{\Ybar}{{\ensuremath{\bar{Y}}}}

\DeclareMathOperator{\rk}{rk}

\DeclareMathOperator{\acl}{acl}   

\DeclareMathOperator{\td}{td}  

\DeclareMathOperator{\loc}{Loc}   

\newcommand{\Qab}{\Q^{\mathrm{ab}}}  
\newcommand{\rad}{\ensuremath{\mathrm{rad}}} 

\DeclareMathOperator{\ldim}{ldim}  
\DeclareMathOperator{\Mat}{Mat}  

\DeclareMathOperator{\etd}{etd}  

\DeclareMathOperator{\ecl}{ecl} 

\newcommand{\N}{\ensuremath{\mathbb{N}}}
\newcommand{\Z}{\ensuremath{\mathbb{Z}}}
\newcommand{\Q}{\ensuremath{\mathbb{Q}}}

\newcommand{\C}{\ensuremath{\mathbb{C}}}

\newcommand{\Rexp}{\ensuremath{\mathbb{R}_{\mathrm{exp}}}}
\newcommand{\Cexp}{\ensuremath{\mathbb{C}_{\mathrm{exp}}}}

\newcommand{\Cat}{\ensuremath{\mathcal{C}}} 
\newcommand{\Loo}{\ensuremath{L_{\omega_1,\omega}}}

\newcommand{\ga}{\ensuremath{\mathbb{G}_\mathrm{a}}}   
\newcommand{\gm}{\ensuremath{\mathbb{G}_\mathrm{m}}}  

\renewcommand{\phi}{\varphi}
\renewcommand{\le}{\ensuremath{\leqslant}}
\renewcommand{\ge}{\ensuremath{\geqslant}}
\newcommand{\tuple}[1]{\ensuremath{\langle #1 \rangle}}
\newcommand{\class}[2]{\ensuremath{\left\{ #1 \,\left|\, #2 \right.\right\}}}

\newcommand{\iso}{\cong}

\newcommand{\into}{\hookrightarrow}

\newcommand{\subs}{\subseteq} 

\newcommand{\minus}{\ensuremath{\smallsetminus}}

\newcommand{\strong}{\ensuremath{\lhd}} 
\newcommand{\nstrong}{\ensuremath{\not\kern-4pt\lhd\;}} 

\newcommand{\gen}[1]{\ensuremath{\left\langle #1 \right\rangle}} 
\newcommand{\hull}[1]{\ensuremath{\lceil #1\rceil}}

\newcommand{\cross}{\ensuremath{\times}}

\newbox\noforkbox \newdimen\forklinewidth
\forklinewidth=0.3pt \setbox0\hbox{$\textstyle\smile$}
\setbox1\hbox to \wd0{\hfil\vrule width \forklinewidth depth-2pt
  height 10pt \hfil}
\wd1=0 cm \setbox\noforkbox\hbox{\lower 2pt\box1\lower
2pt\box0\relax}
\def\unionstick{\mathop{\copy\noforkbox}\limits}

\def\nonfork_#1{\unionstick_{\textstyle #1}}

\setbox0\hbox{$\textstyle\smile$} \setbox1\hbox to \wd0{\hfil{\sl
/\/}\hfil} \setbox2\hbox to \wd0{\hfil\vrule height 10pt depth
-2pt width
                \forklinewidth\hfil}
\newbox\doesforkbox
\setbox\doesforkbox\hbox{\lower 2pt\box1 \lower
2pt\box2\lower2pt\box0\relax}
\def\nunionstick{\mathop{\copy\doesforkbox}\limits}

\def\fork_#1{\nunionstick_{\textstyle #1}}

\newcommand{\ra}[3]{\ensuremath{#1 \stackrel{#2}{\longrightarrow} #3}}

\newcommand{\leteq}{\mathrel{\mathop:}=}

\newarrowmiddle{normal}\lhtriangle\rttriangle\uhtriangle\dhtriangle

\newarrowmiddle{iso}{\cong}{\cong}{}{}   

\newarrow{Partial}{-}{-}{-}{-}{harpoon}     

\newarrow{Equal}=====     

\newarrow{Strong}{hooka}-{normal}->     

\newarrow{Str}{}{}{normal}{}{}     

\newarrow{Iso}--{iso}->     

\newarrow{NT}====>     

\newtheorem{prop}{Proposition}[section]
\newtheorem{cor}[prop]{Corollary}
\newtheorem{theorem}[prop]{Theorem}
\newtheorem{lemma}[prop]{Lemma}

\newtheorem{fact}[prop]{Fact}
\newtheorem*{claim}{Claim}

\theoremstyle{definition}
\newtheorem{defn}[prop]{Definition}
\newtheorem{example}[prop]{Example}

\newtheorem{remark}[prop]{Remark}

\newtheorem{construction}[prop]{Construction}

\renewcommand{\th}{{\ensuremath{{}^{\mathrm{th}}}}}  
\newcommand{\egg}{exponential-graph-generated}
\newcommand{\Kummergen}{Kummer-generic}
\newcommand{\seac}{strongly exponentially-algebraically closed}

\newcommand{\EAC}[1]{\ensuremath{{#1}^\sim}}   
\newcommand{\B}{\ensuremath{\mathbb{B}}}

\title{Finitely Presented Exponential Fields}
\author{Jonathan Kirby}
\date{Version 4.0, \today}
\address{Jonathan Kirby \\
School of Mathematics\\
University of East Anglia\\
Norwich NR4 7TJ\\
UK}

\subjclass[2000]{03C65, 11J81}

\keywords{Exponential fields, Schanuel's conjecture, pseudo-exponentiation, transcendence}

\begin{document}

\begin{abstract}
  The algebra of exponential fields and their extensions is developed. The focus is on ELA-fields, which are algebraically closed with a surjective exponential map. In this context, finitely presented extensions are defined, it is shown that finitely generated strong extensions are finitely presented, and these extensions are classified. An algebraic construction is given of Zilber's pseudo-exponential fields. As applications of the general results and methods of the paper, it is shown that Zilber's fields are not model-complete, answering a question of Macintyre, and a precise statement is given explaining how Schanuel's conjecture answers all transcendence questions about exponentials and logarithms. Connections with the Kontsevich-Zagier, Grothendieck, and Andr\'e transcendence conjectures on periods are discussed, and finally some open problems are suggested.
\end{abstract}

\maketitle

\setcounter{tocdepth}{1}
\tableofcontents

\section{Introduction}

  An \emph{exponential field} (or \emph{E-field}) is a field $F$ of characteristic zero equipped
  with a homomorphism $\exp_F$ (also written $\exp$, or $x \mapsto
  e^x$) from the additive group $\ga(F) = \tuple{F;+}$ to the multiplicative group $\gm(F) = \tuple{F^\times;\cdot}$.
The main examples are the real and complex exponential fields, $\Rexp$ and $\Cexp$, where the exponential map is given by the familiar power series. 

Zilber \cite{Zilber05peACF0} gave axioms describing particular exponential fields which he called ``pseudo-exponential fields''. His construction is model-theoretic and focuses mainly on the uncountable setting. In this paper we develop the algebra of exponential fields leading, amongst other things, to an algebraic construction of the pseudo-exponential fields which gives some more information about them. 

Some of the concepts in this paper appear also in Zilber's paper, but here we present them in a wider and more natural context. In particular, we do not assume that our exponential fields satisfy the Schanuel property, so much of what we do applies unconditionally to the complex setting. The main method of the paper is the use of a predimension function $\delta$. These functions were introduced by Hrushovski \cite{Hru93} for various model-theoretic constructions, but it appears they could have a significant use in transcendence theory as well.

We consider mainly those exponential fields which are algebraically closed and have a surjective exponential map, in this paper called ELA-fields. In \S2 we show that an exponential field $F$, or even a field with a partially-defined exponential map, can be extended in a free way to an ELA-field and, under some extra assumptions such as $F$ being finitely generated, this free extension is unique up to isomorphism. 

In \S3 we give a definition of an extension of ELA-fields being \emph{finitely presented}. The finite presentations take the form of algebraic varieties which are the locus of a suitable generating set. Compare the situation of a finitely generated extension of pure fields, where the extension is determined up to isomorphism by the ideal of polynomials satisfied by the generators, or equivalently by their algebraic locus over the base field. The most important extensions of exponential fields are the so-called \emph{strong extensions}. In \cite{EAEF} it was shown that these are the extensions which preserve the notion of exponential algebraicity. We prove that a finitely generated, kernel-preserving, strong extension of ELA-fields is a finitely presented extension. This theorem could be viewed as the analogue for exponential fields of the Hilbert Basis Theorem, which implies that any finitely generated extension of fields is finitely presented.

The brief \S4 explains the convention for defining finitely presented ELA-fields (as opposed to finitely presented extensions). With this convention, it follows at once that, if Schanuel's conjecture is true, every finitely generated ELA-subfield of $\Cexp$ is finitely presented. We can give a similar, unconditional result. By Theorem~1.2 of \cite{EAEF}, we know that $\Cexp$ is a strong extension of its countable subfield $\C_0$ of exponentially-algebraic numbers. It is therefore an immediate consequence of Theorem~\ref{aleph0-stability} that every finitely generated ELA-extension of $\C_0$ within $\Cexp$ is a finitely presented extension.

In \S5 we show that whether or not a finitely presented extension is strong can be detected from the algebraic variety which gives the presentation, and a classification is given of all finitely generated strong ELA-extensions.

The analogue of the algebraic closure of a field is the strong exponential-algebraic closure $\EAC{F}$ of an exponential field $F$.
Zilber's pseudo-exponential fields are the simplest examples of this construction. The main claim of \cite{Zilber05peACF0} was that the uncountable pseudo-exponential fields are determined up to isomorphism by their cardinality. Unfortunately there is a mistake in the proof there. In the proof of Proposition~5.15 in \cite{Zilber05peACF0}, there is no reason why $A'B'$ should not lie in $C$, and then $V'$ would not contain $V_0$. Indeed, that proposition as stated is false, because the definition of \emph{finitary} used there does not give sufficiently strong hypotheses. The stronger hypotheses of Lemma~5.14 would be enough to prove the main result, but no correct proof is known to me at the time of writing, even with these hypotheses.

In \S6 we construct $\EAC{F}$ and, under some basic assumptions including countability, show that it is unique. In particular, we prove that the countable pseudo-exponential fields are determined up to isomorphism by their exponential transcendence degree. In fact the uniqueness of the pseudo-exponential field $B_{\aleph_1}$ of cardinality $\aleph_1$ then follows by Zilber's methods, as explained  for example in \cite[Theorem~2.1]{OQMEC}, but the higher cardinalities are still problematic. 

In \S7 we answer a question of Macintyre by showing that Zilber's pseudo-exponential fields are not model-complete, and in \S8 we show that ELA-fields satisfying the Schanuel nullstellensatz are not necessarily \seac, in contrast to the situation for pure fields where the Hilbert nullstellensatz characterises algebraically closed fields.

In \S9 we reflect on what the ideas of this paper show for transcendence problems, and try to give a formal statement expressing the generally accepted principle that Schanuel's Conjecture answers all transcendence problems about exponentials and logarithms. We write $\B$ to mean a pseudo-exponential field of cardinality $2^{\aleph_0}$. Zilber's conjecture is that $\Cexp \iso \B$, which on the face of it makes sense only if $\B$ is well-defined, but in fact the conjecture consists of the two assertions that Schanuel's conjecture is true and that $\Cexp$ is \seac, both of which make sense independently of the uniqueness of $\B$. Connections between Schanuel's conjecture and conjectures on periods are explored.

Finally, in \S10 we suggest some open problems.

I am grateful to many people for discussions relating to this paper, particularly to Michel Waldschmidt and Daniel Bertrand for discussions about the relationship with transcendence problems.

\section{Free extensions}

As an intermediate stage in constructing exponential fields we need the notion of a \emph{partial E-field}.
\begin{defn}\label{partial E-field defn}
  A \emph{partial E-field} $F$ consists of a field $\tuple{F;+,\cdot}$ of characteristic zero, a $\Q$-linear subspace $D(F)$ of the additive group of the field, and a homomorphism $\ra{\tuple{D(F);+}}{\exp_F}{\tuple{F;\cdot}}$.

$D(F)$ is the \emph{domain} of the exponential map of $F$, and we write $I(F) = \exp_F(D(F))$, the \emph{image} of the exponential map.

A homomorphism of partial E-fields is a field embedding $\ra{F}{\theta}{F_1}$ such that $\theta(D(F)) \subs D(F_1)$ and for every $x \in D(F)$, $\exp_{F_1}(\theta(x)) = \theta(\exp_F(x))$.
\end{defn}

If $X$ is a subset of a partial E-field $F$, we define the partial E-subfield of $F$ generated by $X$, written $\gen{X}_F$, to have $D(\gen{X}_F)$ equal to the $\Q$-span of $D(F) \cap X$, and the underlying field of $\gen{X}_F$ to be the subfield of $F$ generated by $D(\gen{X}_F) \cup I(\gen{X}_F) \cup X$. Thus $\gen{X}_F$ contains all the exponentials in $F$ of elements of $X$, but does not contain iterated exponentials. A different but equivalent definition of partial E-fields is given in \cite{EAEF}, where $D(F)$ is given as a separate sort. 

In this paper we consider only those partial E-fields $F$ which are algebraic over $D(F) \cup I(F)$.

Now let $F$ be a partial E-field, $\xbar$ a finite tuple from $D(F)$, and $B$ a subset of $D(F)$. We define the \emph{relative predimension function} to be
\[\delta(\xbar/B) = \td(\xbar,\exp(\xbar)/B,\exp(B)) - \ldim_\Q(\xbar/B)\]
where by $\td(X/Y)$ we mean the transcendence degree of the field extension $\Q(XY)/\Q(Y)$ and by $\ldim_\Q(X/Y)$ we mean the dimension of the $\Q$-vector space spanned by $X \cup Y$, quotiented by the subspace spanned by $Y$.
\begin{defn}
An extension $F \subs F_1$ of partial E-fields is \emph{strong}, written $F \strong
F_1$, iff for every tuple $\xbar$ from $D(F_1)$, we have $\delta(\xbar/D(F)) \ge 0$.
  
If $B$ is a subset of $D(F)$, we define $B \strong F$ iff $\gen{B}_F \strong F$.
\end{defn}

As explained in \cite{EAEF}, strong extensions are essentially those for which the notion of exponential algebraicity is preseved, and are thus the most useful extensions to consider. In this paper we see they are intimately connected with free or finitely presented extensions. 

The following basic properties are easy to verify.
\begin{lemma}[Basic properties of $\delta$ and strong extensions]\label{strong lemma} \ 
   \begin{enumerate}
  \item{} [Addition Property] If $\xbar, \ybar \in D(F)$ are finite tuples and $B \subs D(F)$ then
\[\delta(\xbar\cup \ybar/B) = \delta(\ybar/B) + \delta(\xbar/\ybar \cup B)\]
\item Given a finite tuple $\xbar$ from $D(F)$ and $B \subs D(F)$, there is a finite tuple $\bbar$ from $B$ such that $\delta(\xbar/B) = \delta(\xbar/\bbar)$.
  \item The identity $F \subs F$ is strong.
  \item  If $F_1 \strong F_2$ and $F_2 \strong F_3$ then $F_1 \strong
    F_3$. (That is, the composite of strong extensions is strong.)
  \item An extension $F \subs F_1$ is strong iff for every tuple $\xbar$ from $F_1$, the subextension $F \subs \gen{F,\xbar}_{F_1}$ is strong. 
  \item If $F_1 \strong F_2 \strong \cdots \strong F_n \strong \cdots$ is an $\omega$-chain of strong extensions then $F_1 \strong \bigcup_{n < \omega} F_n$.
\item If in addition each $F_n \strong M$ then $  \bigcup_{n < \omega} F_n \strong M$.  \qed
  \end{enumerate}
\end{lemma}

We now explain how exponential maps can be constructed abstractly. Let $F$ be a field of characteristic zero, and $D(F)$ a $\Q$-subspace. We will construct an exponential map defined on $D(F)$.
\begin{construction}\label{E-field construction}
  Choose a \Q-basis $\class{b_i}{i \in I}$ of $D(F)$. For each $i \in I$
  we will choose $c_{i,1} \in F$, and we will define $\exp(b_i) =
  c_{i,1}$.  The value of $\exp(b_i/m)$ must be an $m\th$
  root of $c_{i,1}$, so we have to specify which. Furthermore, as $m$
  varies, we must choose these roots coherently. So in fact for each
  $i \in I$ and $m \in \N$ we much choose $c_{i,m} \in F$ such that
  for any $r,m \in \N$, we have $c_{i,rm}^r = c_{i,m}$. Every element
  of $D(F)$ can be written as a finite sum $\frac{1}{m}\sum r_ib_i$
  for some $m \in \N$ and $r_i \in \Z$, and we define
  $\exp(\frac{1}{m}\sum r_ib_i) = \prod c_{i,m}^{r_i}$. The coherence
  condition shows that $\exp$ is well-defined.
\end{construction}
 This coherence property for the roots is important enough that we introduce some terminology for it.
\begin{defn}
  Given $c_1$, a \emph{coherent system of roots of $c_1$} is a sequence $(c_m)_{m \in \N}$ such that for every $r,m \in \N$ we have $c_{rm}^r = c_m$.
\end{defn}

 Of course, for the exponential map to be non-trivial we need to have some elements other than 1 (and 0) which have $n^{\mathrm{th}}$ roots for all $n$. In this case $F$ will have to be infinite dimensional as a $\Q$-vector space, so there will be a vast number (indeed $2^{|F|}$) of different total exponential maps which can be defined on $F$. Thus, for example there is no hope of classifying or understanding even all the exponential maps on $\Q^{\mathrm{alg}}$.

We will now explain how to construct exponential fields in as \emph{free} a way as possible.
\begin{construction}\label{F^e construction}
  Let $F$ be any partial E-field. We construct an extension $F^e$ of
  $F$ such that $D(F^e) = F$. First, embed $F$ in a large algebraically
  closed field, $\mathcal{C}$. Let $\class{b_i}{i \in I}$ be a
  $\Q$-linear basis for $F/D(F)$. Choose $\class{c_{i,n}}{i \in I, n
    \in \N} \subs \mathcal{C}$ such that the $c_{i,1}$ are
  algebraically independent over $F$, and for each $i$,
  $(c_{i,n})_{n \in \N}$ is a coherent system of roots of $c_{i,1}$. Each $r \in F$ is a finite sum of the form
  $r_0 + \frac{1}{n}\sum m_i b_i$ for some $r_0 \in D(F)$, $n \in \N$,
  and some $m_i \in \Z$, and we define $\exp(r_0 + \frac{1}{n}\sum m_i
  b_i = \exp_F(r_0)\prod c_{i,n}^{m_i}$.  Then let $F^e$ be the subfield
  of $\mathcal{C}$ generated by $F \cup \class{c_{i,n}}{i \in I, n \in
    \N}$.

  A straightforward calculation shows that the isomorphism type of the
  extension $F^e$ of $F$ does not depend on the choice of
  $\mathcal{C}$, the choice of the $b_i$, or the choice of the
  $c_{i,n}$.
\end{construction}
The exponential map on $F^e$ will be a total map only when $F$ is
already a total E-field (and so $F^e = F$). However, we can iterate
the construction to get a total E-field.

\begin{construction}
  We write $F^E$ for the union of the chain
  \[F \into F^e \into F^{ee} \into F^{eee} \into \cdots\] and call it the
  \emph{free (total) E-field extension of $F$}.
\end{construction}

We can also produce E-rings, and algebraically closed E-fields by slight variations on this method. It is convenient (albeit rather ugly) to introduce some terminology for the latter.
\begin{defn}
  An \emph{EA-field} is an E-field whose underlying field is algebraically closed.
\end{defn}
\begin{construction}
  For any partial E-field $F$, let $F^a$ be the algebraic closure
  of $F$, with $D(F^a) = D(F)$.

  We write $F^{EA}$ for the union of the chain
  \[F \into F^e \into F^{ea} \into F^{eae} \into F^{eaea} \into
  \cdots\] and call it the \emph{free EA-field extension} of $F$.
\end{construction}

These constructions can intuitively be seen to be free in that at each stage there are no
unnecessary algebraic or exponentially algebraic relations introduced. In the case of exponential rings (rather than fields), the analogous construction of the free E-ring extension can be seen to have the right category-theoretic universal property of a free object. In \cite{Macintyre91}, a universal property of the free E-field is given in terms of E-ring specializations. The extension $F^{EA}$ has non-trivial automorphisms over $F$, so cannot have a category-theoretic universal property, but later we prove uniqueness statements about these extensions making the intuitive notion of freeness precise.

\subsection*{Logarithms}

A logarithm of an element $b$ of an exponential field $F$ is just some $a$ such that $\exp(a) = b$. Of course such a logarithm will only exist if $b$ is in the image of the exponential map, and will be defined only up to a coset of the kernel. In this algebraic setting there is no topology to make sense of a branch of the logarithm function, as in the complex case. We want to consider exponential fields, like \Cexp, in which every nonzero element has a logarithm, so we extend our terminology conventions.

\begin{defn}
  An \emph{L-field} is a partial exponential field in which every non-zero element has a logarithm. An \emph{EL-field} is a (total) exponential field in which every non-zero element has a logarithm. It is an \emph{LA-field} or \emph{ELA-field} respectively if, in addition, it is algebraically closed.
\end{defn}

The additive group of a field of characteristic zero is just a $\Q$-vector space, whereas the multiplicative group has torsion, the roots of unity, so an L-field must have non-trivial kernel. The most important case is when the kernel is an infinite cyclic group.

\begin{construction}
 Let $\Q_0$ be the partial E-field with underlying field $\Q$, and
 $D(\Q_0) = \{0\}$. Write $\Q^{\mathrm{ab}}$ for the maximal abelian
 extension of $\Q$, given by adjoining all roots of unity. Let
 $\Q^{\mathrm{ab}}(\tau)$ be a field extension with $\tau$ a single
 element, possibly in $\Q^{\mathrm{ab}}$ but non-zero. Let $CK_\tau$ be the partial E-field with underlying field $\Q^{\mathrm{ab}}(\tau)$, with $D(CK_\tau)$ the $\Q$-vector space spanned by $\tau$ and the $\exp(\tau/m)$ forming a coherent system of primitive $m^{\mathrm{th}}$ roots of unity. Then $CK_\tau$ is defined uniquely up to isomorphism by the minimal polynomial of $\tau$ over $\Q$. The letters ``CK'' stand for ``cyclic kernel''. In the special case where $\tau$ is transcendental, we write $SK$ for $CK_\tau$, meaning ``standard kernel''.
\end{construction}

More generally, following Zilber we say that a partial exponential field $F$ has \emph{full kernel} if the image of the exponential map contains the subgroup $\mu$ of all roots of unity (so, in particular, $F$ extends $\Qab$). The next proposition is implicit in \cite{Zilber05peACF0} and shows that the terminology is justified because the property of $F$ having full kernel depends only on the isomorphism type of the kernel of the exponential map as an abelian group.

\begin{prop}
  Let $F$ be a partial E-field extending $\Qab$, and let $K$ be the kernel of its exponential map. Then the following are equivalent.
\begin{enumerate}
  \item $F$ has full kernel
  \item $\Q K/K \iso \mu$
  \item For each $n \in \N^+$, $K/nK$ is a cyclic group of order $n$
  \item For each $n \in \N^+$, $|K/nK| = n$
  \item $\tuple{K;+}$ is elementarily equivalent to $\tuple{\Z;+}$.
\end{enumerate}
Furthermore, if $F$ is a field extending $\Qab$, and $K$ is a subgroup of its additive group which satisfies the equivalent properties (2) --- (5), then there is a partial exponential map on $F$ with kernel $K$.
\end{prop}
We give the proof for the sake of completeness.
\begin{proof}
Note that for $x \in D(F)$, we have $\exp_F(x) \in \mu$ iff $x$ lies in the $\Q$-linear span of the kernel. Thus (1) $\implies$ (2). But also $\mu$ has no proper self-embeddings, so (2) $\implies$ (1).

Consider the ``multiply by $n$ map'' $n : \Q K \to \Q K$. For any $x \in \Q K$, $\exp(x)$ lies in the $n$-torsion of $\Q K/K$ iff $nx \in K$, so the $n$-torsion group of $\Q K/K$ is isomorphic to $n^{-1}K/K$. Since $\Q K$ is divisible and torsion-free, this is isomorphic under the multiply by $n$ map to $K/nK$. But the $n$-torsion of $\mu$ is the cyclic group of order $n$, so we have (2) $\implies$ (3). In fact, $\mu$ is defined up to isomorphism by being a torsion abelian group with this $n$-torsion for each $n$, so (3) $\implies$ (2). Clearly (3) $\implies$ (4). For the converse, it suffices to prove it where $n = p^r$, a prime power. But then we have $p^r$ elements of $K/p^r K$ of order dividing $p^r$ and only $p^{r-1}$ have order dividing $p^{r-1}$, and hence there is an element of order $p^r$, so $K/p^r K$ is cyclic of order $p^r$.

Property (4), together with being a torsion-free abelian group, gives a complete axiomatization of the elementary theory of $\tuple{\Z;+}$ by Szmielew's Theorem \cite[Theorem A.2.7]{Hodges93}, so (4) $\iff$ (5).

For the ``furthermore'' statement, by property (2) there is a homomorphism from $\Q K$ onto $\mu$ with kernel $K$ which makes $F$ into a partial E-field with full kernel.
\end{proof}
 In this paper we are mainly interested in exponential fields with a surjective exponential map, so most partial E-fields we consider will have full kernel. We also assume that extensions of partial E-fields are kernel-preserving (that is, do not add new kernel elements) unless otherwise stated.

Any partial E-field $F$ with full kernel can be extended to an ELA-field without adding new kernel elements. Indeed, we can produce free L-field, LA-field, EL-field, and ELA-field extensions of $F$, written $F^{L}$, $F^{LA}$, $F^{EL}$, and $F^{ELA}$ in analogy to before.
\begin{construction}~\label{ELA-construction}
  Let $F$ be a partial E-field with full kernel. We start by constructing a partial E-field extension $F^l$ of $F$ in which every element of $F$ has a logarithm, and there are no new kernel elements. Embed $F$ in a large algebraically closed field, $\mathcal{C}$. Inside $\mathcal{C}$ we have $F^{\rad}$, the field extension of $F$ obtained by adjoining all roots of all elements of $F$ and iterating this process. The multiplicative group $(F^\rad)^\times$ is divisible, and the image $\exp_F(D(F))$ contains the torsion and is divisible, so the quotient $(F^\rad)^\times/\exp_F(D(F))$ is a $\Q$-vector space.

 Choose $(b_i)_{i \in I}$ from $F$ such that the cosets $b_i \cdot \exp_F(D(F))$ form a $\Q$-linear basis of $(F^\rad)^\times/\exp_F(D(F))$. In other words, the $b_i$ form a multiplicative basis of $(F^\rad)^\times$ over $\exp_F(D(F))$. Now choose $(a_i)_{i \in I}$ from $\mathcal{C}$, algebraically independent over $F$, and for each $i \in I$, choose a coherent system of roots $(b_{i,m})_{m \in \N}$ of $b_i$.

 Let $D(F^l)$ be the $\Q$-subspace of $\mathcal{C}$ spanned by $D(F)$ and the $a_i$. Define $\exp(a_i/m) = b_{i,m}$ and extend the exponential map appropriately. Let $F^l$ be the subfield of $\mathcal{C}$ generated by $D(F^l)$ and $\exp(D(F^l))$. Then every element of $F$ has a logarithm in $F^l$. The isomorphism type of $F^l$ may depend on the choices made, but we write $F^l$ for any resulting partial E-field.

Now we define $F^{ELA}$ to be the union of any chain
 \[F \into F^{e} \into F^{el} \into F^{ela} \into F^{elae} \into F^{elael} \into
  \cdots\] 
iterating the three operations. The chain and its union are not necessarily uniquely defined because the operation $F \mapsto F^l$ is not necessarily uniquely defined. Where the union is uniquely defined we call it the \emph{free ELA-field extension} of $F$. The extensions $F^{L}$, $F^{LA}$, and $F^{EL}$ of $F$ are defined in the obvious way.
\end{construction}

\begin{lemma}\label{F strong in F^ELA}
  For any partial E-field, $F$, the extensions $F \into F^e$, $F \into
  F^a$, $F \into F^E$ and $F \into F^{EA}$ are strong. If $F^l$,
  $F^{ELA}$ are any results of Construction~\ref{ELA-construction}
  then the extensions $F \into F^l$ and $F \into F^{ELA}$ are strong.
\end{lemma}
\begin{proof}
  By construction, for any $\bar{y}$ from $D(F^e)$,
  $\delta(\bar{y}/D(F)) = 0$. Hence $F \strong F^e$. $F \strong F^l$ by the same argument. It is immediate that $F \strong F^a$ because the domain of the exponential map does not extend. The rest follows from Lemma~\ref{strong lemma}.
\end{proof}

In Construction~\ref{F^e construction} of $F^e$ from $F$ we made choices, but in fact the isomorphism type of $F^e$ as an extension of $F$ did not depend on those choices. 
In Construction~\ref{ELA-construction} of $F^l$ and $F^{ELA}$ we again made choices, but in this case the isomorphism types of the extensions do in general depend on those choices. Before giving conditions where the extensions do not depend on the choices, so are well-defined, we illustrate the problem.
 Let $F = CK_\tau^a$, so $D(F)$ is spanned by $\tau$. We want to
 define an extension $F_1$ of $F$ in which 2 has a logarithm. So let
 $F_1 = F(a)$ as a field, with $a$ transcendental over $F$. We define
 $\exp(a/m)$ to be an $m^{\mathrm{th}}$ root of 2. There is no problem
 in doing this, but all of these roots lie in $F$ because it is algebraically closed, so if we make one
 choice of roots and produce $F_1$, and then make a different choice
 of roots and produce $F_2$, then $F_1$ and $F_2$ will not be
 isomorphic as partial E-field extensions of $F$. In fact these
 different choices will all be isomorphic as partial exponential
 fields and even as extensions of $CK_\tau$. The problem is just that
 we had fixed all the roots of 2 in $F$ before we defined the
 logarithms of 2. The way to solve the problem is to put in the
 logarithms earlier in the construction. In fact it is often possible
 to do this because of an important fact about pure fields known as the Thumbtack Lemma. (An explanation of the name can be found in \cite[p19]{Baldwin_Categoricity}.)

The Thumbtack Lemma was proved by Zilber \cite[Theorem~2]{Zilber06covers} (with a correction by Bays and Zilber \cite[Theorem~2.3]{BZ11}). We will give three versions of it in this paper as we need them. All are special cases of the two quoted theorems, but we prefer to state exactly the form we need each time. Given an element $b$ of a field, we write $\sqrt{b}$ for the set of all the $m\th$ roots of $b$ for all $m \in \N$.
\begin{fact}[Thumbtack Lemma, version 1]\label{thumbtack1} \ \\
  Let $F = \Qab(a_1,\ldots,a_r,\sqrt{b_1},\ldots,\sqrt{b_r})$, an extension of $\Qab$ by finitely many generators together with all the roots of some of those generators. Now suppose that $c$ lies in some field extension of $F$ and is multiplicatively independent from $b_1,\ldots,b_r$. Then there is $m \in \N$ and an $m\th$ root $c_m$ of $c$ such that there is exactly one isomorphism type of a coherent system of roots of $c_m$ over $F$. That is, if $F_1$ and $F_2$ are both obtained from $F$ by adjoining $c_m$ and any coherent system of roots of $c_m$, then there is an isomorphism from $F_1$ to $F_2$ over $F$ which sends the chosen system of roots in $F_1$ to the chosen system in $F_2$. 
\end{fact}
Note that if $c$ is transcendental over $F$ then the result is trivial. However, when $c$ is algebraic over $F$ then there is something to prove, and the condition that $c$ is multiplicatively independent of the $b_i$ is essential. Note also that we cannot necessarily take $m=1$. For example, if $F = \Qab$ and $c = 9$ then $F$ certainly knows the difference between $\pm 3$, so we must take $m \ge 2$. Another version of the thumbtack lemma applies to extensions of an algebraically closed field.

\begin{fact}[Thumbtack Lemma, version 2]\ \\
  Let $F = K(a_1,\ldots,a_r,\sqrt{b_1},\ldots,\sqrt{b_r})$, where $K$ is an algebraically closed field of characteristic zero. Suppose that $c$ lies in some field extension of $F$ and is multiplicatively independent from $K^\cross \cdot \langle b_1,\ldots,b_r \rangle$. Then there is $m \in \N$ and an $m\th$ root $c_m$ of $c$ such that there is exactly one isomorphism type of a coherent system of roots of $c_m$ over $F$.
\end{fact}

\begin{defn}
 Let $F \subs F_1$ be an extension of partial E-fields. Then $F_1$ is \emph{finitely generated} as an extension of $F$ iff there is a finite subset $X \subs F_1$ such that $F_1 = \gen{F \cup X}_{F_1}$.

Now let $F$ be an ELA-field, and $X \subs F$ a subset. We define $\gen{X}_F^{ELA}$ to be the smallest ELA-subfield of $F$ which contains $X$. Note that it always exists, as the intersection of ELA-subfields of $F$ is again an ELA-subfield of $F$.
\end{defn}
Note also that $\gen{X}_F^{ELA}$ and $(\gen{X}_F)^{ELA}$ have different meanings. The first is the smallest ELA-subfield of $F$ which contains $X$, and the second is a free ELA-field extension of the smallest partial E-subfield of $F$ containing $X$, which may not be uniquely defined. In favourable circumstances (as below) the latter is well-defined and then the two ELA-fields will sometimes be isomorphic, but neither is generally true.

We now give sufficient conditions on $F$ for $F^{ELA}$ to be
well-defined. For example, from the first case we deduce that
$CK_\tau^{ELA}$ is well-defined. We only consider the case where $F$
is countable here. The general case seems to be more difficult.

\begin{theorem}\label{ELA-well-defined}
  If $F$ is a partial E-field with full kernel which is either finitely generated or a finitely generated extension of a countable LA-field, $F_0$, and $F \strong K$, $F \strong M$ are two strong extensions of $F$ to ELA-fields which do not extend the kernel, then $\gen{F}_K^{ELA} \iso \gen{F}_M^{ELA}$ as extensions of $F$. In particular:
  \begin{enumerate}
  \item The free ELA-closure $F^{ELA}$ of $F$ is well-defined.
\item The extension $F \strong K$ factors as $F \strong F^{ELA} \strong K$.
\end{enumerate}
\end{theorem}

\begin{proof}
  Statements (1) and (2) are immediate from the main part of the theorem. For the main part, enumerate $\gen{F}_K^{ELA}$ as $s_1,s_2,s_3,\ldots$, such that for each $n \in \N$, either
\begin{enumerate}[i)]
 \item $s_{n+1}$ is algebraic over $F \cup \{s_1,\ldots,s_n\}$; or
 \item $s_{n+1} = \exp_K(a)$ for some $a \in F \cup \{s_1,\ldots,s_n\}$; or
 \item $\exp_K(s_{n+1}) = b$ for some $b\in F \cup \{s_1,\ldots,s_n\}$.
\end{enumerate}
This is possible by the definition of $\gen{F}_K^{ELA}$. We will inductively construct chains of partial E-subfields \[F= K_0 \subs K_1 \subs K_2 \subs \cdots\]
 of $K$ and
\[F= M_0 \subs M_1 \subs M_2 \subs \cdots\]
of $M$, and nested isomorphisms $\theta_n: K_n \to M_n$ such that for each $n \in \N^+$ we have $s_n \in K_n$, $K_n \strong K$ and $M_n \strong M$. We also ensure that, as a pure field, each $K_n$ has the form $F_0(\bar{\alpha},\sqrt{\bar{\beta}})$ for some finite tuples $\bar{\alpha},\bar{\beta}$, and $F_0$ either $\Qab$ or a countable algebraically closed field.

We start by taking $\theta_0$ to be the identity map on $F$. Now assume we have $K_n$, $M_n$, and $\theta_n$.
\paragraph{Case 1)} $s_{n+1}$ is algebraic over $K_n$ (including the case where $s_{n+1} \in K_n$). Let $p(X)$ be the minimal polynomial of $s_{n+1}$ over $K_n$. The image $p^\theta$ of $p$ is an irreducible polynomial over $M_n$, so let $t$ be any root of $p^\theta$ in $M$. Let $K_{n+1} = K_n(s_{n+1})$, $M_{n+1} = M_n(t)$, and let $\theta_{n+1}$ be the unique field isomorphism extending $\theta_n$ and sending $s_{n+1}$ to $t$. We make $K_{n+1}$ and $M_{n+1}$ into partial exponential fields by taking the graph of exponentiation to be the graph of $\exp_K$ or $\exp_M$ intersected with $K_{n+1}^2$ or $M_{n+1}^2$ respectively. Suppose that $(a,\exp_K(a)) \in K_{n+1}^2$. Since $K_n \strong K$, we have $\td(a,\exp_K(a)/K_n) - \ldim_\Q(a/D(K_n)) \ge 0$. But $K_{n+1}$ is an algebraic extension of $K_n$, so it follows that $\ldim_\Q(a/D(K_n)) = 0$, that is, that $a \in D(K_n)$. Hence $D(K_{n+1}) = D(K_n)$. The same argument shows that $D(M_{n+1}) = D(M_n)$. Now if $\xbar$ is any tuple from $K$, we have $\delta(\xbar/D(K_{n+1})) = \delta(\xbar/D(K_n)) \ge 0$, and hence $K_{n+1} \strong K$, and similarly $M_{n+1} \strong M$. It is immediate that the pure field $K_{n+1}$ is of the form $F_0(\bar{\alpha},\sqrt{\bar{\beta}})$ because $K_n$ is of that form.

\paragraph{Case 2)} $s_{n+1}$ is transcendental over $K_n$ and $s_{n+1} = \exp_K(a)$ for some $a \in K_n$. Let $K_{n+1} = K_n(\sqrt{s_{n+1}})$ and $M_{n+1} = M_n(\sqrt{\exp_M(\theta_n(a))})$. Extend $\theta_n$ by defining $\theta_{n+1}(\exp_K(a/m)) = \exp_M(\theta_n(a)/m)$, and extending to a field isomorphism. This is possible because $s_{n+1}$ is transcendental over $K_n$ and $\exp_M(\theta_n(a))$ is transcendental over $M_n$ (the latter because $M_n \strong M$), and so there is a unique isomorphism type of a coherent system of roots of $s_{n+1}$ over $K_n$, and of $\exp_M(\theta_n(a))$ over $M_n$. Then $\td(K_{n+1}/K_n) = 1$, $a \in D(K_{n+1}) \minus D(K_n)$, and $K_n \strong K$, so $D(K_{n+1})$ is spanned by $D(K_n)$ and $a$. Similarly, $D(M_{n+1})$ is spanned by $\theta_n(a)$ over $D(M_n)$, so $\theta_{n+1}$ is an isomorphism of partial E-fields.

Now if $\xbar$ is any tuple from $K$, we have 
\[\delta(\xbar/D(K_{n+1})) = \delta(\xbar,a/D(K_n)) - \delta(a/D(K_n)) = \delta(\xbar,a/D(K_n)) - 0 \ge 0\]
as $K_n \strong K$, so $K_{n+1} \strong K$. The same argument shows that $M_{n+1} \strong M$. Again, it is immediate that the pure field $K_{n+1}$ is of the required form.

\paragraph{Case 3)} $s_{n+1}$ is transcendental over $K_n$, not of the form $\exp_K(a)$ for any $a \in K_n$, but $\exp_K(s_{n+1}) = b$ for some $b \in K_n$. By hypothesis, $K_n$ has the form $F_0(\bar{\alpha},\sqrt{\bar{\beta}})$ for some finite tuples $\bar{\alpha},\bar{\beta}$, and $F_0$ either $\Qab$ or a countable algebraically closed field. Hence, by either version~1 or version~2 of the Thumbtack Lemma, there is $N \in \N^+$ and $c$ such that $c^N = b$ and there is a unique isomorphism type of a coherent sequence of roots of $c$ over $K_n$. Let $t \in M$ be such that $\exp_M(t) = \theta_n(c)$. Let $K_{n+1} = K_n(s_{n+1},\sqrt{c})$ and $M_{n+1} = M_n(t, \sqrt{\theta_n(c)})$. Extend $\theta_n$ by defining $\theta_{n+1}(s_{n+1}) = Nt$, and $\theta_{n+1}(\exp_K(s_{n+1}/Nm)) = \exp_M(t/m)$, and extending to a field isomorphism. This is possible by the choice of $N$, the fact that $s_{n+1}$ is transcendental over $K_n$ and (since $M_n \strong M$) the fact that $t$ is transcendental over $M_n$. As in Case 2 above, we have $K_{n+1} \strong K$, $M_{n+1} \strong M$, and the pure field $K_{n+1}$ of the required form.

\paragraph{Conclusion}
That completes the induction. Let $K_\omega = \bigcup_{n \in \N} K_n$. Then $K_\omega = \gen{F}_K^{ELA}$ because $K_\omega$ is an ELA-subfield of $K$ containing $F$ and is the smallest such because at each stage we add only elements of $K$ which must lie in every ELA-subfield of $K$ containing $F$. The union of the maps $\theta_n$ gives an embedding of $K_\omega$ into $M$, and, for the same reason, the image must be $\gen{F}_M^{ELA}$. Hence $\gen{F}_K^{ELA} \iso \gen{F}_M^{ELA}$ as required.
\end{proof}

\section{Finitely presented extensions}

 We say that a partial E-field $F$ is \emph{finitely generated} if there is a finite subset $X$ of $F$ such that $F = \gen{X}_F$. We restrict now to those partial E-fields $F$ which are generated as fields by $D(F) \cup I(F)$ (call them \emph{\egg}). Similarly, an ELA-field $F$ is finitely generated as an ELA-field if $F = \gen{X}_F^{ELA}$ for some finite subset $X$ of $F$. An extension $F \subs F_1$ of ELA-fields is finitely generated iff there is a finite subset $X$ of $F_1$ such that $F_1 = \gen{F \cup X}_{F_1}^{ELA}$, and similarly for partial E-fields.

Let $F \subs F_1$ be a finitely generated extension of \egg\ partial E-fields, say generated by $a_1,\ldots,a_n \in D(F_1)$. Then the isomorphism type of the extension is given by the algebraic type of the infinite tuple $(\bar{a},\exp(\bar{a}/m))_{m \in \Z}$ over $F$. Let
 \[I(\bar{a}) = \class{f \in F[\bar{X},(Y_{m,i})_{m \in \Z, i=1,\ldots,n}]}{f(\bar{a},(e^{a_i/m}))
  = 0}\]
 and for $m \in \N^+$, let 
\[I_m(\bar{a}) = \class{f \in F[\bar{X},\bar{Y}_m, \Ybar_m^{-1}]}{f(\bar{a}/m,e^{\bar{a}/m}, e^{-\bar{a}/m}) = 0}.\]

The ideal $I(\bar{a})$ contains all the \emph{coherence polynomials} of the form $Y_{mr,i}^r - Y_{m,i}$ for each $m,r \in \Z$, and each $i = 1,\ldots,n$, which force each sequence $e^{a_i}, e^{a_i/2},e^{a_i/3},\ldots$ to be a coherent system of roots. Including the negative powers ensures that they are nonzero. The ideal $I(\abar)$ determines all the ideals $I_m(\abar)$ and if $m_1$ divides $m_2$ then $I_{m_1}(\bar{a})$ is determined by $I_{m_2}(\bar{a})$.

\begin{defn}
  An ideal $I$ of the polynomial ring $F[\bar{X},(Y_{m,i})_{m \in \Z, i=1,\ldots,n}]$ is \emph{additively free} iff it does not contain any polynomial of the form $\sum_{i=1}^n r_iX_i - c$ with the $r_i \in \Z$, not all zero, and $c \in D(F)$. It is \emph{multiplicatively free} iff it does not contain any polynomial of the form $\prod_{i=1}^n Y_{1,i}^{r_i} - d$ with the $r_i \in \Z$, not all zero, and $d \in \exp(D(F))$. Similarly we say that $I_m$ is additively free or multiplicatively free if it does not contain any polynomials of these forms.
\end{defn}
If $F \subs F_1$ is a finitely generated extension of \egg\ partial E-fields, we may choose the generators $a_1,\ldots,a_n$ to be $\Q$-linearly independent over $D(F)$, and this corresponds to the ideal $I(\abar)$ being additively free. Conversely, if $I$ is any prime ideal of the polynomial ring $F[\bar{X},\bar{Y}_1,\bar{Y}_2,\bar{Y}_3,\ldots]$ which contains the coherence polynomials and is additively free, then it defines an extension $F_I$ of $F$, the field of fractions of the ring $F[\bar{X},\bar{Y}_1,\bar{Y}_2,\bar{Y}_3,\ldots]/I$, with exponentiation defined in the obvious way. All we have really done is translated Construction~\ref{E-field construction} into the language of ideals.
\begin{lemma}
  If $I$ is a prime ideal containing the coherence polynomials and is additively free, then the extension $F_I$ it defines has the same kernel as $F$ iff $I$ is multiplicatively free.
\end{lemma}
\begin{proof}
 Write $a_i$ for the image of $X_i$ in $F_I$. If $I$ is not multiplicatively free then for some $r_i \in \Z$, not all zero, and some $c \in D(F)$, we have $\prod_{i=1}^n e^{r_ia_i} = e^c$, so $c-\sum_{i=1}^n r_ia_i$ lies in the kernel of $\exp_{F_I}$. Since $I$ is additively free, this element does not lie in $D(F)$, in particular it does not lie in the kernel of $\exp_{F}$. Conversely, if $I$ is multiplicatively free and $\exp(c + 1/m\sum_{i=1}^n r_i a_i) = 1$ with $c \in D(F)$ and $m,r_i \in \Z$, then $\prod_{i=1}^n\exp(a_i)^{r_i} = \exp(c)^m$, so $r_i = 0$ for each $i$, and $c$ lies in the kernel of $\exp_{F}$.
\end{proof}

\begin{defn}
  We say that an extension $F \subs F_1$ of partial E-fields is \emph{finitely presented} iff it has a finite generating set $a_1,\ldots,a_n$, which is \Q-linearly independent from $D(F)$, such that $I(\abar)$ is generated as an ideal by the coherence polynomials together with a finite set of other polynomials.
\end{defn}
\begin{defn}
  An additively free prime ideal $J$ of $F(\Xbar,\Ybar_1,\Ybar_1^{-1})$ is said to be \emph{\Kummergen} iff there is only one additively free prime ideal $I$ of $F[\bar{X},(Y_{m,i})_{m \in \Z, i=1,\ldots,n}]$, containing the coherence polynomials, such that $J = I_1$, as defined above.
\end{defn}
The term \emph{Kummer-generic} is due to Martin Hils, \cite[p10]{Hils10}. The usage here is not exactly the same as in that paper, because there they consider only adding new points to the multiplicative group, whereas here we are adding $\abar$ to the additive group as well as $e^\abar$ to the multiplicative group. The connection with Kummer theory can be seen from \cite[Lemma~5.1]{BZ11}.

\begin{lemma}\label{rootdet lemma}
  If $F \subs F_1$ is a finitely presented extension of partial E-fields, then it has a generating set $a_1',\ldots,a_n'$ such that the ideal $I_1(\abar')$ is \Kummergen.
\end{lemma}
\begin{proof}
 Let $g_1,\ldots,g_r \in I(\abar)$, together with the coherence polynomials, be a generating set for $I(\abar)$. Let $N \in \N$ be the least common multiple of the $m$ such that some variable $Y_{m,i}$ occurs in some $g_j$. Then $I(\abar)$ is determined by $I_N(\abar)$. Take $\abar' = \abar/N$, so $I_1(\abar') = I_N(\abar)$. Then $I_1(\abar')$ is \Kummergen, as required.
\end{proof}

\begin{example}
Take an extension of an EA-field $F$ generated by $a_1,a_2$, such
  that $e^{a_1/2} = a_2$, $e^{a_2} = a_1+1$. Then 
  \[I_1 = \gen{Y_{1,1} = X_2^2,\; Y_{1,2} = X_1 + 1}\]
  and 
  \[I_2 = \gen{Y_{2,1} = X_2,\; Y_{2,2}^2 = X_1 + 1.}\]

 In this case, $I_1$ is not \Kummergen\ because it does not resolve
  whether $e^{a_1/2} = \pm a_2$.
\end{example}

There are finitely generated kernel-preserving extensions of some partial E-fields which are not finitely presented. However, another version of the thumbtack lemma gives conditions when this pathology does not occur.
\begin{fact}[Thumbtack Lemma, version 3]
  Let $F = K(a_1,\ldots,a_r,\sqrt{b_1},\ldots,\sqrt{b_r})$, where $K$ is an algebraically closed field of characteristic zero. Suppose that $c_1,\ldots,c_n$ lie in some field extension of $F$ and are multiplicatively independent from $K^\cross \cup \{b_1,\ldots,b_r\}$. Then there is $N \in \N$ and $N\th$ roots $c'_i$ of $c_i$ such that there is exactly one isomorphism type over $F$ of an $n$-tuple of coherent systems of roots of the $(c_i')$.
\end{fact}
As an immediate corollary, we have:
\begin{cor}\label{fg pE extensions are fp}
  If $F$ is an LA-field, $F_1$ is a finitely generated partial E-field
  extension of $F$, and $F_2$ is a finitely generated partial E-field
  extension of $F_1$, which does not extend the kernel, then $F_2$ is
  a finitely presented extension of $F_1$. In particular, every
  finitely generated kernel-preserving partial E-field extension of an
  ELA-field is finitely presented. \qed
\end{cor}

Our main interest is not with partial E-fields, but with ELA-fields. 
\begin{defn}
  A finitely generated extension $F \subs F_1$ of countable ELA-fields is said to be \emph{finitely presented} iff there is a finite generating set $\abar$ such that, taking $K=\gen{F,\abar}_{F_1}$, the partial E-field extension of $F$ generated by $\abar$, we have $F_1 \iso K^{ELA}$.
\end{defn}
Note that $K^{ELA}$ is well-defined by Proposition~\ref{ELA-well-defined}. From Construction~\ref{E-field construction} it is clear that most finitely generated extensions of ELA-fields are not finitely presented. Indeed there are only countably many finitely presented extensions of a given countable ELA-field, but $2^{\aleph_0}$ finitely generated extensions.

We introduce a notation for finitely presented extensions. Since these are given by \Kummergen\ ideals $I_1$, which are ideals in a polynomial ring with finitely many indeterminates, we can consider instead their associated varieties as subvarieties of $(\ga \cross \gm)^n$.
\begin{defn}
  Let $F$ be an ELA-field. An irreducible subvariety $V$ of $(\ga \cross \gm)^n$ defined over $F$ is said to be \emph{additively free, multiplicatively free, and \Kummergen} iff the corresponding ideal $I(V)$ is. 
\end{defn}
Suppose that $V$ satisfies all three conditions. Then there is a uniquely determined partial E-field extension $K$ of $F$ which is generated by a tuple $(\abar,e^{\abar})$ which is generic in $V$ over $F$. We write $F|V$, read ``$F$ extended by $V$'', for the ELA-extension $K^{ELA}$ of $F$.

\begin{theorem}\label{aleph0-stability}
  Let $F \strong F_1$ be a finitely generated kernel-preserving strong extension of ELA-fields. Then $F_1$ is a finitely presented extension of $F$.
\end{theorem}

\begin{proof}
  Let $\abar$ be a finite set of generators of $F_1$ over $F$. By extending $\abar$ if necessary, we may assume that $\gen{F,\abar}_{F_1} \strong F_1$. By Corollary~\ref{fg pE extensions are fp}, the extension $F \subs \gen{F,\abar}_{F_1}$ is a finitely presented extension of partial E-fields. By Theorem~\ref{ELA-well-defined}, $F_1 \iso (\gen{F,\abar}_{F_1})^{ELA}$, so is a finitely presented ELA-extension of $F$.
\end{proof}

Note that there are finitely generated strong extensions of partial E-fields, of E-fields, of EA-fields, and of EL-fields which are not finitely presented, due to the issue of uniqueness of coherent systems of roots. This is the main reason why we work with ELA-fields. It is also important that the kernel does not extend, as if $a$ is a new kernel element then the values of $\exp(a/m)$ for $m \in \N^+$ cannot all be specified by a finite list of equations.

\section{Finitely Presented ELA-fields}

So far we have defined finitely presented \emph{extensions} of ELA-fields, but it is natural also to ask about finitely presented ELA-fields. The useful convention is as follows.
\begin{defn}\label{fp defn}
An ELA-field $F$ is said to be \emph{finitely presented} iff there is a finitely generated partial E-field $F_0$ (with full kernel) such that $F = F_0^{ELA}$.
\end{defn}
Note that a finitely presented ELA-extension of a finitely presented ELA-field is still finitely presented, since if $F = F_0^{ELA}$, 
$V \subs (\ga \cross \gm)^n$ is defined over $F$, additively free, multiplicatively free, and \Kummergen, and $(\abar,e^\abar) \in V$ generates the extension $F \subs F|V$ and $F_1 =  \gen{F_0 \cup \abar}_{F|V}$ then $F|V \iso F_1^{ELA}$.

The definition is just a convention since there is no way to specify any partial E-field with full kernel within the category of partial E-fields, just by finitely many equations between a given set of generators. Within the subcategory of partial E-fields with full kernel, one might view the $CK_\tau$ as finitely presented, with explicit finite presentations
\[\exp(\tau/2) = -1, \qquad f(\tau)=0\]
where $f$ is the minimal polynomial of the cyclic generator $\tau$. However it does not follow that $\exp(\tau/m)$ is a primitive $m\th$
root of 1 for each $m$ and this cannot be specified by finitely many polynomial equations, for example $\tau/p$ could be the cyclic generator for any odd prime $p$. On the other hand, within the category of partial E-fields with cyclic kernel and named generator $\tau$, the minimal polynomial of $\tau$ does indeed determine $CK_\tau$ precisely. So the matter of what constitutes a finite presentation is somewhat dependent on the axioms specifying the category, and the convention in Definition~\ref{fp defn} is the useful one for the purposes of this paper.

\section{Classification of strong extensions}

It follows from Theorem~\ref{aleph0-stability} that finitely generated kernel-preserving strong extensions of ELA-fields are all of the form $F \strong F|V$ where $V$ is additively and multiplicatively free, and \Kummergen. We next discuss the properties of the varieties $V$ which occur in this way.

Let $G = \ga \cross \gm$. Each matrix $M \in \Mat_{n \cross n}(\Z)$ defines a homomorphism
$\ra{G^n}{M}{G^n}$ by acting as a linear map on $\ga^n$ and as a
multiplicative map on $\gm^n$. If $V \subs G^n$, we write $M\cdot V$
for its image. Note that if $V$ is a subvariety of $G^n$, then so is $M\cdot
V$.
\begin{defn}
  An irreducible subvariety $V$ of $G^n$ is \emph{rotund} iff for
  every matrix $M \in \Mat_{n \cross n}(\Z)$
  \[\dim M\cdot V \ge \rk M.\]
  A reducible subvariety is rotund iff at least one of its
  irreducible components is rotund.

  A subvariety $V$ of $G^n$ is \emph{perfectly rotund} iff it is
  irreducible, $\dim V = n$, for every $M \in \Mat_{n \cross
    n}(\Z)$ with $ 0 < \rk M < n$,
  \[\dim M\cdot V \ge \rk M + 1\]
and also $V$ is additively free, multiplicatively free, and \Kummergen.
\end{defn}
Note that a reducible subvariety may satisfy the dimension property of
rotundity without being rotund. For example, take $n=2$, $V_1$ given by $x_1
= y_1 = 1$, $V_2$ given by $x_2 = y_2 = 1$, and $V = V_1 \cup V_2$.

\begin{prop}\label{rotund prop}
  Let $F \subs F|V$ be an extension of ELA-fields, with $V$ additively and multiplicatively free, and \Kummergen. Then the extension is strong iff $V$ is rotund.
\end{prop}
\begin{proof}
  Let $\abar$ be the tuple generating $F|V$ over $F$ such that $(\abar,e^\abar) \in V$.
  Suppose $F \strong F|V$, let $M \in \Mat_{n\cross n}(\Z)$, and let
  $\bar{b} = M\bar{a}$. Then $\loc_F(\bar{b},e^{\bar{b}}) = M\cdot
  V$ and $\ldim_\Q(\bar{b}) = \rk M$, so 
  \[\dim M \cdot V - \rk M = \delta(\bar{b}/F) \ge 0\]
  and hence $V$ is rotund.

  Conversely, suppose that $V$ is rotund, let $F_1 = \gen{F,\abar}_{F|V}$, the partial E-field extension of $F$ generated by $\abar$, and let $\bar{b}$ be any
  tuple from $D(F_1)$. The tuple $\bar{a}$ spans $D(F_1)/F$, so
  there is $M \in \Mat_{n\cross n}(\Z)$ such that there is an equality of $\Q$-vector spaces $\gen{M\bar{a}}/F
  = \gen{\bar{b}}/F$. Then 
  \[\delta(\bar{b}/F) = \delta(M\bar{a}/F) = \dim M\cdot V - \rk M
  \ge 0\] so $F \strong F_1$. But $F|V = F_1^{ELA}$, so $F \strong F|V$ as required.
\end{proof}

\begin{defn}
  A strong extension $F \strong F_1$ of ELA-fields is \emph{simple} iff
  whenever $F \strong F_2 \strong F_1$ is an intermediate ELA-field then $F_2 = F$ or $F_2 = F_1$.
\end{defn}
It is easy to see that every simple extension of ELA-fields is finitely generated. For, suppose $\abar$ is a non-empty finite tuple from $F_1 \minus F$. Then there is a finite tuple $\abar'$, extending $\abar$, such that $\gen{F,\abar'}_{F_1} \strong F_1$. Then $F \strong F_2 \leteq \gen{F,\abar'}^{ELA}_{F_1}$ and $F_2 \strong F_1$, so by simplicity $F_2 = F_1$ and the extension is finitely generated. However, simple extensions are not necessarily generated by a single element.

It is important to distinguish between exponentially algebraic and exponentially transcendental extensions. The full definition of exponential algebraicity is given in \cite{EAEF}, but all we will use is the following fact:
\begin{fact}[{\cite[Theorem~1.3]{EAEF}}]\label{etd fact}
  Let $F$ be an E-field and suppose $C \strong F$ is some strong subset, and $\abar$ is a finite tuple from $F$. Then the exponential transcendence degree of $\abar$ over $C$ in $F$ satisfies \[\etd^F(\abar/C) = \min\class{\delta(\abar,\bbar/C)}{\bbar \mbox{ is a finite tuple from } F}.\]
\end{fact}
Exponential transcendence degree is the dimension notion of a pregeometry, analogous to transcendence degree in pure fields. An element $a$ is exponentially algebraic over $C$ iff $\etd^F(a/C) = 0$.

\begin{lemma}\label{unique generic extensions}
  There is a unique simple exponentially transcendental extension of any ELA-field.
\end{lemma}
\begin{proof}
  Let $F \strong F_1$ be simple, with $a \in F_1$, exponentially transcendental over $F$. Then $\td(a,e^a/F) = 2$, so the partial E-field extension $\gen{F,a}_{F_1}$ is determined uniquely up to isomorphism. But $\gen{F,a}_{F_1} \strong F_1$ by the above characterization of exponential transcendence degree so, by Theorem~\ref{ELA-well-defined}, $\gen{F,a}^{ELA}_{F_1} \iso (\gen{F,a}_{F_1})^{ELA}$. Then $\gen{F,a}^{ELA}_{F_1} \strong F_1$, so $\gen{F,a}^{ELA}_{F_1} = F_1$ because the extension is simple.
\end{proof}
Note that if $a$ is exponentially transcendental over $F$ then
$\loc(a,e^a/F) = G$ (recall that $G =\ga\cross \gm$), so the simple exponentially transcendental extension of $F$ can be written in our notation as $F|G$.

\begin{prop}\label{simple = perf rotund}
If $V$ is perfectly rotund then the strong extension of ELA-fields $F \strong F|V$ is simple and exponentially-algebraic.

Conversely, if $F \strong F'$ is a simple, exponentially-algebraic extension of ELA-fields then $F' \iso_F F|V$ for some perfectly rotund $V$.
\end{prop}
\begin{proof}
 Let $\abar$ be the tuple generating $F|V$ over $F$ such that $(\abar,e^\abar) \in V$, and let $F_1 = \gen{F,\abar}_{F|V}$, the partial E-field extension of $F$ generated by $\abar$.

 Since $V$ is rotund and additively and multiplicatively free, $F \strong F|V$ is exponentially algebraic iff $\dim V = n$. Now suppose $F \strong F_2 \strong F|V$, a proper
  intermediate ELA-field. Then $(F_2 \cap D(F_1))/F$ is a non-trivial
  proper subspace of $D(F_1)/F$, which must be the span of
  $M\bar{a}$ for some $M \in \Mat_{n \cross n}(\Z)$, with $0 < \rk M <
  n$. Since $V$ is rotund, $\dim M\cdot V \ge \rk M$. Extend
  $M\bar{a}$ to a spanning set $M\bar{a},\bar{b}$ of $D(F_1)/F$.
  Then $\delta(\bar{b}/F, M\bar{a}) \ge 0$, because $F_2 \strong
  F|V$. But
  \begin{eqnarray*}
    \delta(\bar{b}/F, M\bar{a}) & = & \td(\bar{b},e^{\bar{b}}/F,
    M\bar{a}, e^{M\bar{a}}) - \ldim_\Q(\bar{b}/F, M\bar{a})\\
    &=& [n - \dim M\cdot V] - [n - \rk M]
  \end{eqnarray*}
  so $\dim M \cdot V \le \rk M$. Thus $\dim M \cdot V = \rk M$, and
  $V$ is not perfectly rotund.

For the converse, choose $\abar$ a tuple of smallest length which generates $F'$ over $F$ and such that $F_1 \leteq \gen{F,\abar}_{F'} \strong F'$, and let $V = \loc(\abar,e^\abar/F) \subs G^n$. Then $F' \iso F_1^{ELA}$. Since $n$ is minimal, $V$ is additively and multiplicatively free. By replacing $\abar$ by $\abar /m$ for some $m\in \N$, we may assume $V$ is \Kummergen. Since the extension is strong and exponentially algebraic, $V$ is rotund and $\dim V = n$. If $V$ is not perfectly rotund then there is a matrix $M$ with $0 < \rk M < n$ such that
  $\dim M \cdot V = \rk M$. Let $F_2 = \gen{F,M\bar{a}}_{F'}^{ELA}$. Then
  $F \strong F_2 \strong F'$, but $F_2 \neq F$ and $F \strong F'$ is simple, so $F_2 = F'$. But $F_2$ is generated by $M\bar{a}$, which is $\Q$-linearly dependent over $F$, so by a basis for it which is a tuple shorter than $\abar$. This contradicts the choice of $\abar$. So $V$ is perfectly rotund.
\end{proof}

We now consider the problem of when two extensions $F|V$ and $F|W$ are isomorphic. Suppose $\abar$ is a generator of $F|V$, with $(\abar,e^\abar) \in V$. Then if $\bbar$ is a different choice of basis of the extension, so $F \cup \bbar$ has the same $\Q$-linear span as $F \cup \abar$, and $W = \loc(\bbar,e^\bbar/F)$, then clearly $F|W \iso F|V$, but there is no reason why $W$ should be equal to $V$. Essentially this is the only way an isomorphism can happen.

\begin{defn}
Suppose $V \subs G^n$, $W \subs G^m$ are two perfectly rotund varieties, defined over $F$. Write $V \sim_F W$ iff $n=m$, there are $M_1,M_2 \in \Mat_{n\cross n}(\Z)$ of rank $n$, and there is $\cbar \in F^n$, such that $M_1 \cdot V = M_2 \cdot W + (\cbar,e^\cbar)$ (where $+$ means the group operation in $G^n$, so multiplication on the $\gm$ coordinates), and furthermore $M_1 \cdot V$ is Kummer-generic.
\end{defn}

\begin{prop}
If $V$ and $W$ are perfectly rotund and defined over $F$ then $F|V \iso_F F|W$ iff $V \sim_F W$.
\end{prop}
\begin{proof}
Firstly suppose that $V \sim_F W$, and let $V' = M_1 \cdot V$ where $M_1$ is as above. Let $K = \gen{F,\abar}_{F|V}$, where $(\abar,e^\abar) \in V$ is the generating tuple. Let $\bbar = M_1 \abar$. Then $\gen{F,\bbar}_{F|V} = K$, and $(\bbar,e^\bbar)$ is generic in $V'$. Furthermore, since $V'$ is Kummer-generic (by assumption), $K$ is well-defined by $V'$. Hence $F|V \iso F|V'$. Similarly, translating $V'$ to $V' - (\cbar, e^\cbar)$ for some $\cbar \in F^n$ does not change $K$. So $F|V \iso_F F|W$.

Conversely, suppose $F|V \iso F|W$. Let $(\abar,e^\abar) \in V$ be a generating tuple for $F|V$. Let $F_1 = \gen{F,\abar}_{F|V}$, and write $F|V$ as the union of a chain of partial E-fields
\[F \strong F_1 \strong F_2 \strong F_3 \strong \cdots \]
where for $n \in \N^+$ we have $\ldim_\Q(D(F_{n+1})/D(F_n)) = 1$, which is possible since $F|V = F_1^{ELA}$.
There is $\bbar \in F|V$ such that $\loc(\bbar,e^\bbar/F) = W$. Suppose that $\bbar$ is  $\Q$-linearly independent over $D(F_1)$. Then, since $F_1 \strong F|V$ we have $\loc(\bbar,e^\bbar/F_1) = W$. Now each $D(F_{n+1})$ is generated over $D(F_n)$ by a single element $c_{n+1}$ such that either $c_{n+1}$ or $e^{c_{n+1}}$ is algebraic over $F_n$. By perfect rotundity of $W$, no $b$ in the $\Q$-linear span of $\bbar$ satisfies this, so inductively we see that $\bbar$ is linearly independent over $D(F_n)$ for all $n$, a contradiction. So $\bbar$ is not \Q-linearly independent over $D(F_1)$. Write $B$ for the $\Q$-linear span of $F \cup \bbar$. Then $D(F_1) \strong F|V$ and $B \strong F|V$, so $D(F_1) \cap B \strong F|V$, and hence, since $V$ and $W$ are perfectly rotund, we must have $B = D(F_1)$, and $V \sim_F W$ as required.
\end{proof}

We can give a normal form for a finitely generated strong ELA-extension $F \strong F'$. The key is that the order of making simple extensions can often be interchanged.
\begin{lemma}\label{order of extensions}
  Let $F$ be a countable ELA-field, and $V \subs G^n$, $W \subs G^r$ two additively and multiplicatively free, irreducible, \Kummergen\ subvarieties, defined over $F$. Then 
  \[(F|V)|W \iso (F|W)|V \iso F|(V\cross W)\]
   as extensions of $F$.
\end{lemma}
\begin{proof}
First note that the extension $(F|V)|W$ makes sense, since in the base change from $F$ to $F|V$, the variety $W$ remains free, irreducible, and Kummer-generic, because both $F$ and $F|V$ are algebraically closed. Similarly $(F|W)|V$ makes sense.
  Now let $\abar, \bbar$ be the tuples in $F_1 = (F|V)|W$ such that $(\abar,e^\abar) \in V$ determines the first extension and $(\bbar,e^\bbar) \in W$ determines the second extension. Similarly, let $\abar',\bbar'$ be the equivalent tuples in $F_2 = (F|W)|V$. Then the partial E-fields $K_1 = \gen{F\abar,\bbar}_{F_1}$ and $K_2 = \gen{F\abar',\bbar'}_{F_2}$ are isomorphic extensions of $F$, because both $(\abar,e^\abar,\bbar,e^\bbar)$ and $(\abar',e^{\abar'},\bbar',e^{\bbar'})$ are generic in $V \cross W$ over $F$. Now $K_1 \strong F_1$ and $K_2 \strong F_2$, hence the result follows by Theorem~\ref{ELA-well-defined}.
\end{proof}
Indeed the extensions $F|V$ and $F|W$ can be freely amalgamated over $F$, and the free amalgam is in fact given by $F|(V\cross W)$.

Now consider a finitely generated strong extension of countable ELA-fields $F \strong F'$. Let $\abar_1$ be some tuple from $F'$, $\Q$-linearly independent over $F$, such that $V_1 \leteq \loc(\abar_1,e^{\abar_1}/F)$ is perfectly rotund. If it does not exists, then $F = F'$. So we have $F = F_0 \strong F_1 = \gen{F,\abar_1}^{ELA}_{F'} \strong F'$. Now iteratively choose tuples $\abar_i$, $\Q$-linearly independent over $F_{i-1}$, such that $V_i \leteq \loc(\abar_i,e^{\abar_i}/F)$ is perfectly rotund and defined over $F_j$ where $j$ is as small as possible, and define $F_i = \gen{F_{i-1},\abar_i}^{ELA}_{F'}$. At some finite stage we will exhaust $F'$. The previous propositions show there is only a very limited scope for choosing the tuples $\abar_i$. Thus we have a Jordan-H\"older-type theorem, showing how a finitely generated extension decomposes as a chain of simple extensions, and the extent to which the chain is unique.
\begin{theorem}\label{Jordan-Holder}
If $F \strong F'$ is a finitely generated strong extension of countable ELA-fields, then it can be decomposed as:
\[F = K_0 \strong K_1 \strong K_2 \strong \cdots \strong K_r = F'\]
such that $K_i = K_{i-1}|V_i$ with $V_i =  V_{i,1} \cross \cdots V_{i,m_i}$
with each $V_{i,j}$ perfectly rotund and defined over $K_{i-1}$ but not defined over $K_{i-2}$. Furthermore, if there is another decomposition
\[F = K_0 \strong K_1' \strong K_2' \strong \cdots \strong K_s' = F'\]
such that $K_i' = K_{i-1}'|V_i'$ with $V_i' =  V_{i,1}' \cross \cdots V_{i,q_i}'$
then $s=r$ and, for each $i$, $q_i = m_i$ and there is a permutation $\sigma$ of $\{1,\ldots,m_i\}$ such that $V_{i,j}' \sim_{F'} V_{i,\sigma(j)}$. \qed
\end{theorem}
A finer analysis is possible, in which one takes into account for each $V_{i,j}$ precisely which of the $V_{s,t}$ for $s<i$ are involved in the field of definition of $V_{i,j}$, to produce a partial order on the simple extensions.

\section{The strong exponential-algebraic closure}

We now consider the analogue for exponential fields of the algebraic closure of a field. 

\begin{defn}
  An exponential field $F$ is said to be \emph{strongly exponentially-algebraically closed} iff it is an ELA-field and for every finitely generated partial E-subfield $A$ of $F$, and every finitely generated kernel-preserving exponentially algebraic strong partial E-field extension $A \strong B$, there is an embedding $B \into F$ fixing $A$.
\end{defn}

The word \emph{strongly} in this definition actually does not refer to
the strong extensions, but rather signifies that the property is
stronger than another property, called \emph{exponential-algebraic
  closedness}, which was also considered by Zilber. Exponential-algebraic
closedness is a model-theoretic approximation to strong
exponential-algebraic closedness which is first-order axiomatizable,
but strong exponential-algebraic closedness is the sensible notion
from the algebraic point of view taken in this paper.

We now show that every countable ELA-field has a well-defined strong
exponential-algebraic closure. 
\begin{theorem}\label{EA-closures}
  Let $F$ be a countable ELA-field. Then there is a \seac\ $\EAC{F}$ with $F \strong
  \EAC{F}$ such that if $F \strong K$, $K$ is strongly exponentially-algebraically closed, and $\ker(K) = \ker(F)$ then there is a strong embedding $\EAC{F} \strong K$ such that 
\begin{diagram}[small]
  F & \rStrong & \EAC{F}\\
    & \rdStrong & \dStrong \\
    &            & K
\end{diagram}
commutes. Furthermore, $\EAC{F}$ is unique up to isomorphism as an
extension of $F$. We call it the \emph{strong exponential-algebraic closure} of $F$.
\end{theorem}
The key property we need is the amalgamation property, which follows from Lemma~\ref{order of extensions}.
\begin{proof}[Proof of Theorem~\ref{EA-closures}]
  Let $F$ be a countable ELA-field. List the triples
  $(n_\alpha,V_\alpha,A_\alpha)_{\alpha < \omega}$ such that
  $n_\alpha \in \N^+$, $V_\alpha$ is a perfectly rotund subvariety of $G^{n_\alpha}(F)$, $A_\alpha$ is a finitely generated subfield of $F$ over which $V_\alpha$ is defined, and $F$ does not contain $\bbar$ such that $(\bbar,e^\bbar)$ is generic in $V_\alpha$ over $A_\alpha$. Note that if $F$ is not \seac\ then there will be $\aleph_0$ such triples.

Let $F_1$ be the ELA-extension of $F$ obtained by making the simple ELA-extensions determined by each $V_\alpha$ in turn. By Lemma~\ref{order of extensions} (and a back and forth argument), $F_1$ does not depend on the choice of well-ordering. Now iterate the process to produce a chain
\[F \strong F_1 \strong F_2 \strong F_3 \strong \cdots\]
and let $\EAC{F}$ be the union of this chain.

By construction, $\EAC{F}$ is strongly exponentially-algebraically closed, and $F \strong \EAC{F}$. Furthermore if $F$ is \seac\ then $F = \EAC{F}$. The primality property and the uniqueness of $\EAC{F}$ follow from a standard back-and-forth argument.
\end{proof}
If $F$ is a partial E-field such that $F^{ELA}$ is well-defined, then
we also write $\EAC{F}$ for $\EAC{(F^{ELA})}$.

Note that if $\EAC{F} \neq F$ then $\EAC{F}$ will not be \emph{minimal} over $F$, that is, it will be isomorphic over $F$ to a proper subfield of itself. This is because we adjoin $\aleph_0$ copies of the extension of $F$ defined by $V_1$ in constructing $F_1$, and if we miss out co-countably many of those realisations, we get a proper ELA-subfield of $F_1$ which is isomorphic over $F$ to $F_1$.

\begin{defn}
  Zilber's \emph{pseudo-exponential fields} are precisely the exponential fields $F$ satisfying the following properties:
 \begin{enumerate}
 \item $F$ is an ELA-field;
 \item $F$ has standard kernel;
 \item The Schanuel property holds; 
 \item $F$ is \seac;
 \item For any countable subset $A \subs F$, the exponential-algebraic closure $\ecl^F(A)$ is countable. 
\end{enumerate}
\end{defn}
Of course these are genuine exponential fields in our algebraic sense. The prefix ``pseudo'' refers to Zilber's programme of pseudo-\emph{analytic} or pseudo-\emph{complex} structures. 

\begin{construction}[Zilber's pseudo-exponential fields]\label{pseudoexp}
Let $B_0 = \EAC{SK}$. For each ordinal $\alpha$, define $B_{\alpha+1} = \EAC{(B_\alpha|G)}$. For limit $\alpha$, take unions. It is easy to see that the exponential transcendence degree of $B_\alpha$ is $|\alpha|$, and that the isomorphism type of
$B_\alpha$ depends only $|\alpha|$. For a cardinal $\kappa$ we write $B_\kappa$ for the model of exponential transcendence degree $\kappa$.  
\end{construction}

By construction, the $B_\kappa$ satisfy Zilber's axioms, and hence are pseudo-exponential fields. Although $B_\kappa$ exists for all cardinals $\kappa$, we have only proved that $F^{ELA}$ and hence $\EAC{F}$ are  uniquely defined for countable $F$, and hence the arguments of this paper only show that $B_\kappa$ is well-defined for countable $\kappa$.

We now proceed to examine strong exponential-algebraic closedness in more detail before proving that the $B_\kappa$ for countable $\kappa$ are the only countable pseudo-exponential fields.

The property of strong exponential-algebraic closedness is most useful
when there is a proper subset of $F$ which is strongly embedded in
$F$, and especially when a finite such subset exists.
\begin{defn}
  An E-field $F$ is said to have ASP, \emph{Almost the Schanuel
    Property}, iff there is a finite tuple $\cbar$ from $F$ such that
    $\gen{\cbar}_F \strong F$.
\end{defn}
\begin{lemma}
Any strong extension of a finitely presented ELA-field has ASP.
\end{lemma}
\begin{proof}
If $F$ is a strong extension of a finitely presented ELA-field, then it is a strong extension of a finitely generated partial E-field $F_0$, and we can take $\cbar$ to be a generating tuple for $F_0$.
\end{proof}

\begin{example}
  Consider the exponential field $\C_{2^x} = \tuple{\C;+,\cdot,2^x}$. Then $\C_{2^x}$ does not satisfy the Schanuel property because $2^1 = 2$, but if Schanuel's conjecture is true then it does satisfy ASP.
\end{example}
ASP is a slight weakening of the Schanuel Property which allows for some extra generality such as this example, but such that the theory works almost unchanged.

\begin{lemma}\label{seac characterization}
 Suppose $F$ is an ELA-field. Then the following are equivalent.
 \begin{enumerate}
  \item[(1)] $F$ is strongly exponentially-algebraically closed.
  \item[(2)] For each $n \in \N$, every perfectly rotund subvariety $V \subs G^n(F)$, and every finitely generated subfield $A$ of $F$ over which $V$ is defined, there is $(\bbar,\exp(\bbar))$ in $F$, generic in $V$ over $A$.
  \item[(3)] For each $n \in \N$, every additively and multiplicatively free, rotund subvariety $V \subs G^n(F)$, and every finitely generated subfield $A$ of $F$ over which $V$ is defined, there are infinitely many distinct $(\bbar,\exp(\bbar))$ in $F$, generic in $V$ over $A$.
 \end{enumerate}
Furthermore, if $F$ satisfies ASP then the next three conditions are also equivalent to the first three.
\begin{enumerate}
 \item[(4)] For each $n \in \N$, every perfectly rotund subvariety $V \subs G^n(F)$, and every finitely generated ELA-subfield $A$ of $F$ over which $V$ is defined, there is $(\bbar,\exp(\bbar))$ in $F$, generic in $V$ over $A$.
\item[(5)] For each finitely-generated ELA-subfield $A$ of $F$, and each finitely generated exponentially-algebraic strong ELA-extension $A \strong B$, there is an embedding $B \into F$ fixing $A$.
\item[(6)] For each finitely-generated strong ELA-subfield $A \strong F$, and each simple exponentially-algebraic strong ELA-extension $A \strong B$, there is an embedding $B \into F$ (necessarily strong) fixing $A$.
\end{enumerate}
\end{lemma}
\begin{proof}
  (1) $\iff$ (2) by Proposition~\ref{simple = perf rotund}. (3)
     $\implies$ (2) is trivial. To show (2) $\implies$ (3), first note
     that every finitely generated strong extension is the union of a
     chain of simple strong extensions, so to find a point in an
     additively and multiplicatively free rotund subvariety it is
     enough to find points in perfectly rotund subvarieties. Now we
     show by induction on $r \in \N$ that there are at least $r$ many
     such $\bbar$. The case $r=1$ is (2). Now suppose we have
     $\bbar_1,\ldots,\bbar_r$. Then by (2) there is a $\bbar_{r+1}$
     such that $(\bbar_{r+1},\exp(\bbar_{r+1}))$ is generic in $V$
     over $A \cup
     \{\bbar_1,\exp(\bbar_1),\ldots,\bbar_r,\exp(\bbar_r)\}$ In
     particular, $\bbar_1,\ldots,\bbar_{r+1}$ are distinct.

It is immediate that (4) implies (2), that (4) implies (5), and that (5) implies (6). We now assume that there is a finite $\cbar \strong F$. Assume (2), and let $A = \gen{\abar,\cbar}_F^{ELA}$ be a finitely generated ELA-subfield of $F$. Since $\cbar \strong F$, we may assume that $A \strong F$ by extending the tuple $\abar$ if necessary. By (2), there is $(\bbar,\exp(\bbar))$ in $F$, generic in $V$ over $\abar$. By Lemma~\ref{order of extensions}, the ELA-subfield $\gen{\abar,\bbar}^{ELA}_F$ of $F$ is isomorphic to $A|V$, and $(\bbar,\exp(\bbar))$ is generic in $V$ over $A$. Hence (4) holds.

Now assume (6), let $V$ be perfectly rotund, and let $A$ be a finitely-generated ELA-subfield over which $V$ is defined. Then there is a finitely-generated ELA-subfield $A'$ of $F$ containing $A$ and $\cbar$ such that $A' \strong F$. By (6), there is a realisation of $A'|V$ in $F$ over $A'$, say generated by $(\bbar,\exp(\bbar))$, generic in $V$ over $A'$. But then $(\bbar,\exp(\bbar))$ is generic in $V$ over $A$ as $A \subs A'$, hence (4) holds.
\end{proof}

\begin{prop}\label{etd determines extension}
  Let $F_0$ be a finitely generated partial E-field with full kernel (or a finitely presented ELA-field), and let $F_0 \strong F$ be a countable, kernel-preserving, \seac\ strong extension of $F_0$. Then $F$ is determined up to isomorphism as an extension of $F_0$ by the exponential transcendence degree $\etd(F/F_0)$. 
\end{prop}
\begin{proof}
Suppose $F_0$ is as above and let $\Cat(F_0)$ be the category of all countable strong kernel-preserving ELA-extensions of $F_0$, with strong embeddings fixing $F_0$ as the morphisms. Let $\Cat^{<\aleph_0}(F_0)$ be the full subcategory of finitely generated ELA-extensions of $F_0$. Then $\Cat(F_0)$ is an $\aleph_0$-amalgamation category, that is:
\begin{itemize}
  \item Every arrow is a monomorphism;
  \item $\Cat_0(F)$ has unions of $\omega$-chains (by Lemma~\ref{strong lemma});
  \item $\Cat^{<\aleph_0}(F_0)$ has only countably many objects up to isomorphism (by Theorem~\ref{aleph0-stability});
  \item For each $A \in \Cat^{<\aleph_0}(F_0)$, there are only countably many extensions of $A$ in $\Cat^{<\aleph_0}(F_0)$ up to isomorphism (also by Theorem~\ref{aleph0-stability});
  \item $\Cat^{<\aleph_0}(F_0)$ has the amalgamation property (by Lemma~\ref{order of extensions}); and
  \item $\Cat^{<\aleph_0}(F_0)$ has the joint embedding property (since $F_0^{ELA}$ embeds in all of the strong ELA-extensions of $F_0$, by Theorem~\ref{ELA-well-defined}).
\end{itemize}  
 Thus by the Fra\"iss\'e amalgamation theorem, specifically the version in \cite[Theorem~2.18]{TEDESV}, there is a unique extension $F_0 \strong F$ in $\Cat(F_0)$ which is $\Cat^{<\aleph_0}(F_0)$-saturated, that is, such that for any finitely generated ELA-extension $A$ of $F_0$ inside $F$, and any finitely generated strong ELA extension $A \strong B$, there is a (necessarily strong) embedding of $B$ into $F$ over $A$. Using part (6) of Lemma~\ref{seac characterization}, this property is the same as being \seac\ together with $\etd(F/F_0)$ being infinite. Thus the proposition is proved in the case where $\etd(F/F_0) = \aleph_0$.
 
Now suppose $F_0 \strong F$ is as in the proposition with $\etd(F/F_0) = n \in \N$. Let $a_1,\ldots,a_n$ be an exponential transcendence base for $F$ over $F_0$, and let $F_1 = \gen{F_0,a_1,\ldots,a_n}^{ELA}_F$. Then $F_1 \iso_{F_0} F_0|G^n$, and $\etd(F/F_1) = 0$. So it is enough to consider the case $\etd(F/F_0) = 0$. Let $\Cat_0(F_0)$ be the subcategory of $\Cat(F_0)$ consisting of the exponentially-algebraic extensions. The same argument as above shows that $\Cat(F_0)$ is an $\aleph_0$-amalgamation category, and we deduce that up to isomorphism there is a unique countable, kernel-preserving, \seac\ strong extension of $F_0$, which of course is $\EAC{F_0}$.
\end{proof}

\begin{cor}
The countable pseudo-exponential fields are precisely $B_\kappa$ for $\kappa$ a countable cardinal. 
\end{cor}
\begin{proof}
The pseudo-exponential fields are all kernel-preserving \seac\ strong extensions of $SK$.
\end{proof}


From the proof of Proposition~\ref{etd determines extension} one can see that much of the machinery of $\aleph_0$-stable first-order theories can be brought to bear on the category $\Cat(F_0)$ for any finitely presented ELA-field $F_0$. Indeed, the \seac\ kernel-preserving strong extensions of $F_0$ (at least those with the countable closure property) should be thought of as analogous to the algebraically closed pure field extensions of a finitely generated field. They are the ``universal domains'' which are saturated and $\aleph_0$-homogeneous for the category $\Cat(F_0)$. Of course this is not saturation nor $\aleph_0$-stability in the sense of first-order model theory, because we are only considering extensions which do not extend the kernel. Also the $\aleph_0$-stability is with respect to counting types over strong ELA-subfields of $F$, not over arbitrary subsets. Developing the homogeneity property further, we make some observations about extending automorphisms.

\begin{prop}\label{extending autos} Suppose that $F$ is a partial E-field with full kernel, which is finitely generated or a finitely generated extension of a countable LA-field, and that $\sigma$ is an automorphism of $F$. Then:
 \begin{enumerate}
  \item $\sigma$ extends uniquely to an automorphism of $F^E$.
  \item $\sigma$ extends to automorphisms of $F^{EA}$, $F^{ELA}$, and to any countable \seac\ kernel-preserving strong extension of $F$, including $\EAC{F}$.
 \end{enumerate}
\end{prop}
\begin{proof}
To extend an automorphism $\sigma$ of $F$ to an automorphism of $F^e$ we must have $\sigma(c_{i,n}) = \exp(\sigma(b_i)/n)$, in the notation of Construction~\ref{F^e construction}. This does define a partial automorphism since the $c_i$ are algebraically independent over $F$, and it extends uniquely to an automorphism of $F^e$ because $F^e$ is generated over $F$ by the $c_{i,n}$. Thus, by induction, $\sigma$ extends uniquely to an automorphism of $F^E$.

We have an extension $F \rTo^{\theta} F^{ELA}$ where $\theta$ is the inclusion map, and a second extension $F \rTo^{\theta \circ \sigma} F^{ELA}$. The partial E-field $F$ satisfies the hypothesis of Theorem~\ref{ELA-well-defined}, so by that theorem there is a map $F^{ELA} \rTo^{\bar{\sigma}} F^{ELA}$ which restricts to $\sigma$ on $F$. The image of $\bar{\sigma}$ is an ELA-subfield of $F^{ELA}$ containing $F$, so must be all of $F^{ELA}$. Hence $\bar{\sigma}$ is an automorphism of $F^{ELA}$ extending $\sigma$. The restriction of $\bar{\sigma}$ to $F^{EA}$ is an automorphism of $F^{EA}$ extending $\sigma$. Similarly, we can use the $\Cat^{<\aleph_0}(F_0)$-saturation and $\Cat_0^{<\aleph_0}(F_0)$-saturation properties to extend $\bar{\sigma}$ from $F^{ELA}$ to an automorphism of $\EAC{F}$ or of another countable \seac\ kernel-preserving strong extension of $F$.
\end{proof}

The partial E-field $SK$ embeds in $\Cexp$, so the restriction, $\sigma_0$, of complex conjugation is an automorphism of $SK$, and it is easy to see that it and the identity map are the only automorphisms of $SK$. By Proposition~\ref{extending autos}, $\sigma_0$ extends to automorphisms of $B_\kappa$ for any countable $\kappa$, and in \cite{KMO10}, these extensions of $\sigma_0$ are used to identify the algebraic numbers which are pointwise definable in pseudo-exponential fields. However, the extensions of $\sigma_0$ are far from being unique, so this argument does not give an analogue of complex conjugation on any $B_\kappa$.

\section{Non-model completeness}
In this section we show that the $B_\kappa$, and other \seac\ E-fields, are not model complete. We use the submodularity property of $\delta$ which is well-known from Hrushovski's amalgamation constructions, and some simple consequences.

\begin{lemma}[Submodularity] 
Let $F$ be a partial E-field, and let $C,X,Y$ be \Q-subspaces of $D(F)$ such that $C \subs X \cap Y$ and $\ldim_\Q(X\cup Y/C)$ is finite. Let $\xbar,\ybar,\zbar$ be finite tuples such that $\xbar \cup C$ spans $X$, $\ybar \cup C$ spans $Y$, and $\zbar \cup C$ spans $X \cap Y$. Then
\begin{equation}
\delta(\xbar \cup\ybar/C) + \delta(\zbar/C) \le \delta(\xbar/C) + \delta(\ybar/C) 
\end{equation}
More prosaically, the predimension function $\delta(-/C)$ is submodular on the lattice of $\Q$-linear subspaces of $D(F)$.
We write 
\[\delta(XY/C) + \delta(X \cap Y/C) \le \delta(X/C) + \delta(Y/C).\]
\end{lemma}
\begin{proof}
Note that if $\delta$ is replaced by $\td$ then (1) holds, and if $\delta$ is replaced by $\ldim_\Q$ then it holds with $\le$ replaced by $=$. Subtracting the latter from the former gives (1).
\end{proof}

\begin{lemma}\label{hull lemma}
  Suppose $F$ is a partial E-field, and $C \strong F$. For each finite tuple $\abar$ from $F$, there is a smallest $\Q$-vector subspace $\hull{C,\abar}_F$ of $F$ containing $\abar$ and $C$, called the \emph{hull} of $C \cup\abar$, such that $\hull{C,\abar}_F \strong F$. Furthermore, $\hull{C,\abar}_F$ is finite-dimensional as an extension of $C$.
\end{lemma}
\begin{proof}
  Since $C \strong F$, there is a finite-dimensional extension $C \subs A \subs D(F)$ with $\abar \in A$ such that $\delta(A/C)$ is minimal, say equal to $d$. If $A_1$ and $A_2$ are two such, then by submodularity we see that $\delta(A_1\cap A_2/C) \le d$, and hence we can take $\hull{C,\abar}_F$ to be the intersection of all such $A$.
\end{proof}

\begin{remark}
Often in amalgamation-with-predimension constructions, the analogue of what is here called the \emph{hull} is called the \emph{strong closure} or, when \emph{self-sufficient} is used in place of strong, the \emph{self-sufficient closure}. While the notion of a self-sufficient subset makes semantic sense ($X$ is self-sufficient in $F$ if no witnesses outside $X$ are needed to realise its full type in $F$), the sense is lost when dealing with extensions rather than subsets because in ``$F$ is a self-sufficient extension of $X$'', the ``self'' should semantically refer to $X$ rather than $F$, in conflict with the syntactic construction of the phrase. Since the focus here is on extensions rather than subsets, we do not use the terminology of self-sufficiency. Similarly, the terminology ``strong closure'' conflicts with the notion here of strong exponential-algebraic closure. The simplest amalgamation-with-predimension construction is that of the universal acylic graph, and there the concept corresponding to our hull is exactly the convex hull of a set in the sense of the graph, that is, the hull of $X$ is the union of all paths between elements of $X$. 
\end{remark}

\begin{prop}\label{almost strong}
  Suppose $F$ is an E-field, $C \strong F$, and $\abar$ is a tuple from $F$. Suppose $K \subs F$ is an E-subfield of $F$, containing $C$, such that $\hull{C,\abar}_F \cap K$ is spanned by $C \cup \abar$. Let $r = \delta(\abar/C) - \delta(\hull{C,\abar}_F/C)$. Then $\etd^K(\abar/C) = \etd^F(\abar/C) + r$.
\end{prop}
\begin{proof}
  By Fact~\ref{etd fact} and the definition of the hull,
\[\etd^K(\abar/C) = \min\class{\delta(\abar,\bbar/C)}{\bbar \subs K} = \delta(\hull{C,\abar}_K/C)\]
and similarly
\[\etd^F(\abar/C) = \delta(\hull{C,\abar}_F/C).\]
So we must show that $\delta(\hull{C,\abar}_K/C) = \delta(\abar/C)$, or equivalently that for any $\bbar$ from $K$, $\delta(\abar,\bbar/C) \ge \delta(\abar/C)$.

  Let $A$ be the $\Q$-span of $C \cup \abar$, $H = \hull{\abar,C}_F$, and let $B \subs K$ be an extension of $A$, generated by some tuple $\bbar$. Then, by the assumption on $K$, $B \cap H = A$. By the submodularity of $\delta$, 
\[\delta(B/C) - \delta(A/C) \ge \delta(BH/C) - \delta(H/C)\]
but the right hand side is positive as $H = \hull{A}_F$. Hence $\delta(B/C) \ge \delta(A/C)$ as required.
\end{proof}

\begin{prop}\label{non-model-complete}
Let $F$ be a countable strongly exponentially-algebraically closed E-field
satisfying ASP, and of exponential transcendence degree at least 1. Then there
is $K \subs F$, a proper E-subfield such that $K \iso F$ but the
inclusion $K \into F$ is not an elementary embedding. In particular,
$F$ is not model-complete.
\end{prop}

\begin{proof}\footnote{My thanks to Alf Onshuus who noticed a mistake in an earlier version of this proof.}
Using the ASP assumption, let $\cbar$ be a finite tuple such that $\cbar \strong F$ and $\etd(F/\cbar) \ge 1$, and let $F_0 = \gen{\cbar}^{ELA}_F$. So $F_0$ is a finitely generated strong ELA-subfield of $F$.

The precise variety $V$ we use is not so important so we take a simple example, the intersection of three generic hyperplanes in $G^3$. That is, let $\alpha_1, \ldots, \alpha_{18} \in F_0$ be algebraically independent and let $V$ be the subvariety of $G^3$ given by
\begin{eqnarray*}
\alpha_1 X_1 + \alpha_2 X_2  + \alpha_3 X_3 +\alpha_4 Y_1  + \alpha_5 Y_2 + \alpha_6 Y_3 = 1 \\
\alpha_7 X_1 + \alpha_8 X_2  + \alpha_9 X_3 +\alpha_{10} Y_1  + \alpha_{11} Y_2 + \alpha_{12} Y_3 = 1 \\
\alpha_{13} X_1 + \alpha_{14} X_2  + \alpha_{15} X_3 +\alpha_{16} Y_1  + \alpha_{17} Y_2 + \alpha_{18} Y_3 = 1 \\
\end{eqnarray*}
where $X_1,X_2,X_3$ are the coordinates in $\ga$ and $Y_1,Y_2,Y_3$ are the coordinates in $\gm$. 
\begin{claim}
 $V$ is perfectly rotund.
\end{claim}
\begin{proof}
Certainly $V$ is irreducible and has dimension 3. The projections to $\ga^3$ and to $\gm^3$ are dominant, so $V$ is additively and multiplicatively free. Similarly, for any $M \in \Mat_{3\cross 3}(\Z)$, if $\rk M = 2$ then $\dim M\cdot V = 3$ and if $\rk M = 1$ then $\dim M\cdot V = 2$. $V$ must be \Kummergen\ from its simple structure, but in any case we could replace it by the locus of $(X_1/m,X_2/m,X_3/m,\sqrt[m]{Y_1},\sqrt[m]{Y_1},\sqrt[m]{Y_1})$ for a suitably large integer $m$ (where $(X_1,X_2,X_3,Y_1,Y_2,Y_3)$ is generic in $V$) without affecting the rest of the argument.
\end{proof}

Choose $(a_1,a_2,a_3) \in F^3$ such that $(a_1,a_2,a_3,e^{a_1},e^{a_2},e^{a_3})$ is generic in $V$ over $F_0$. Since $F$ is strongly exponentially-algebraically closed and has ASP, such a point exists by Lemma~\ref{seac characterization}. Now let $t \in F$ be exponentially transcendental over $F_0$, let $F_1 = \gen{F_0(t)}_F^{ELA}$, and let $K_1 = \gen{F_0(a_1)}_F^{ELA}$.

\begin{claim}
 $a_2,a_3 \notin K_1$.
\end{claim}
\begin{proof}
The intuition here is that $V$ already gives the maximum three constraints between $a_1$, $a_2$, and $a_3$. If $a_2$ or $a_3$ were to lie in $K_1$ that would be an extra constraint, or perhaps $r+1$ extra constraints with $r$ extra witnesses, which would contradict $F_0$ being strong in $F$.

Suppose for a contradiction that $a_2 \in K_1$. Then there is a shortest chain of subfields of $K_1$
\[\acl^F(F_0(a_1,e^a_1)) = L_0 \subset L_1 \subset L_2 \subset \cdots \subset L_r\]
such that $a_2 \in L_r$ and, for each $i \in \{1,\ldots,r\}$, there are $x_i,e^{x_i} \in L_i$ such that $L_i = \acl^F(L_{i-1}(x_i,e^{x_i}))$ and either $x_i \in L_{i-1}$ or $e^{x_i} \in L_{i-1}$.

For each $i \in \{1,\ldots,r\}$, we have $\td(L_i/L_{i-1}) = 1$. We can consider each $L_i$ as a partial exponential field by taking the intersection of the graph of $\exp_F$ with $L_i^2$. Then for each $i$, $L_i \neq L_{i-1}$, so $x_i \in D(L_i) \minus D(L_{i-1})$, so in particular $a_1,x_1,\ldots,x_r$ are $\Q$-linearly independent over $F_0$, and 
$e^{a_1},e^{x_1},\ldots,e^{x_r}$ are multiplicatively independent over $F_0$.

Let $V'\subs G^2$ be the fibre of $V$ given by fixing the coordinates $X_1=a_1$ and $Y_1=e^{a_1}$. So $V'$ is the locus of $(a_2,a_3,e^{a_2},e^{a_3})$ over $L_0$. Also $\dim V' = 1$, and $V'$ projects dominantly to each coordinate, so $a_2,a_3,e^{a_2},e^{a_3}$ are interalgebraic over $L_0$. In particular, they all lie in $L_r$, and their locus over $L_{r-1}$ is $V'$. Since $V$ is additively and multiplicatively free, so is $V'$. So $a_2,a_3$ are $\Q$-linearly independent over $L_{r-1}$ and $e^{a_2},e^{a_3}$ are multiplicatively independent over $L_{r-1}$.

Thus if $x_r \in L_{r-1}$ then $a_1,a_2,a_3,x_1,\ldots,x_r$ are \Q-linearly independent over $F_0$. Otherwise, $e^{x_r} \in L_{r-1}$, and $e^{a_1},e^{a_2},e^{a_3},e^{x_1},\ldots,e^{x_r}$ are multiplicatively independent over $F_0$, but then again (since the kernel of the exponential map lies in $F_0$) $a_1,a_2,a_3,x_1,\ldots,x_r$ are \Q-linearly independent over $F_0$.

So we have 
\begin{multline*}
\td(a_1,a_2,a_3,x_1,\ldots,x_r,e^{a_1},e^{a_2},e^{a_3},e^{x_1},\ldots,e^{x_r}/F_0) \\
=\td(L_r/L_0) + \td(L_0/F_0) = r + 2 
\end{multline*}

and thus
\[\delta(a_1,a_2,a_3,x_1,\ldots,x_r/F_0) = r+2 - (r+3) = -1\]
which contradicts $F_0 \strong F$. Hence $a_2,a_3 \notin K_1$.
\end{proof}
Indeed, the proof of the claim shows that $a_2,a_3$ must be $\Q$-linearly independent over $K_1$ since their locus over $K_1$ is the same as over $L_0$. Now $\hull{F_0,a_1}^F$ is spanned by $F_0,a_1,a_2,a_3$, so by Proposition~\ref{almost strong}, 
\[\etd^{K_1}(a_1/F_0) = \etd^{F}(a_1/F_0) + \delta(a_1/F_0) - \delta(a_1,a_2,a_3/F_0) = 0+1-0 = 1.\]
Thus $\etd^{K_1}(a_1/F_0) = \etd^{F_1}(t/F_0) = 1$, so there is an isomorphism
$\theta_1:F_1 \to K_1$ taking $t$ to $a_1$ and fixing $F_0$ pointwise. Now choose an $\omega$-chain of ELA-subfields of $F$
\[F_1 \strong F_2 \strong F_3 \strong \cdots \strong F\]
such that $F_{n+1}$ is a simple strong ELA-extension of $F_n$, for each $n$, and $\bigcup_{n \in \N}F_n = F$. Inductively we construct chains of ELA-subfields $(K_n)_{n \in \N}$ of $F$ and isomorphisms $\theta_n:F_n \to K_n$, and we also prove that $\etd(F/F_n) +1 = \etd(F/K_n)$. (If $\etd(F)$ is infinite this is trivially true since both sides will be equal to $\aleph_0$.) We already have $K_1$ and $\theta_1$. Note that $\etd(F/F_1) + 1 = \etd(F/K_1)$ since $\{t\}$ is an exponential transcendence base for $\ecl^F(F_1)$ over $\ecl^F(K_1)$.

Suppose we have $K_n$ and $\theta_n$. If $F_n \strong F_{n+1}$ is an exponentially transcendental simple extension, then choose any $b \in F$ which is exponentially transcendental over $K_n$, and take $K_{n+1} = \gen{K_n(b)}^{ELA}_F$. This $b$ exists because $\etd(F/F_n) \le \etd(F/K_n)$. Also $\etd(F/F_{n+1}) + 1 = \etd(F/F_n)$ and $\etd(F/K_{n+1}) + 1 = \etd(F/K_n)$, so $\etd(F/F_{n+1}) +1 = \etd(F/K_{n+1})$. By Lemma~\ref{unique generic extensions}, $\theta_n$ extends to an isomorphism $\theta_{n+_1}: F_{n+1} \to K_{n+1}$. 
Now suppose that $F_n \strong F_{n+1}$ is exponentially algebraic. Let $V_{n+1}$ be the corresponding perfectly rotund subvariety, say given by some equations $f_i = 0$, with coefficients in $F_n$. Let $W_{n+1}$ be the subvariety obtained from $V_{n+1}$ by applying $\theta_n^{-1}$ to all the coefficients of the $f_i$. Then $W_{n+1}$ is a perfectly rotund subvariety defined over $K_n$, and $K_n$ is a finitely generated ELA-subfield of $F$, which satisfies ASP, so by Lemma~\ref{seac characterization} there is a realization of the ELA-extension of $K_n$ corresponding to $W_{n+1}$ in $F$. Let $K_{n+1}$ be such a realization, and let $\theta_{n+1}$ be any isomorphism from $F_{n+1}$ to $K_{n+1}$ extending $\theta_n$. Also $\etd(F/F_{n+1}) = \etd(F/F_n)$ and $\etd(F/K_{n+1}) = \etd(F/K_n)$.

Let $K = \bigcup_{n \in \N}K_n$ and $\theta = \bigcup_{n \in \N} \theta_n$. Then $\theta: F \to K$ is an isomorphism. But the inclusion $K \subs F$ is not an elementary embedding, because
\[ F \models \exists x_2,x_3 [(a_1,x_2,x_3,e^{a_1},e^{x_2},e^{x_3}) \in V] \]
but $K_1 \strong K$, so $\etd^K(a_1) = 1$, and hence
\[ K \models \neg \exists x_2,x_3 [(a_1,x_2,x_3,e^{a_1},e^{x_2},e^{x_3}) \in V] \]
\end{proof}

 Note that we have shown more than non-model-completeness in the language of exponential fields. We have
 shown that even if one adds symbols for every definable subset of the
 kernel, then the result is still not model-complete.

We can now prove:
\begin{theorem}\label{pexp not model complete}
 Zilber's pseudo-exponential fields (of exponential transcendence degree at least 1) are not model-complete.
\end{theorem}
\begin{proof}
  Proposition~\ref{non-model-complete} shows that $B_\kappa$ is not model-complete when $1 \le \kappa \le \aleph_0$. By \cite[Theorem~5.13]{Zilber05peACF0}, every pseudo-exponential field of infinite exponential transcendence degree is $\Loo$-equivalent to $B_{\aleph_0}$, so in particular elementarily equivalent, and hence also not model-complete.
\end{proof}

\section{The Schanuel nullstellensatz}

D'Aquino, Macintyre and Terzo \cite{DMT10} and also Shkop \cite{Shkop} have shown that every strongly exponentially-algebraically closed  exponential field satisfies the \emph{Schanuel nullstellensatz}: 
\begin{defn}
  An ELA-field $F$ is said to satisfy the Schanuel nullstellensatz iff whenever $f \in F[X_1,\ldots,X_n]^E$ is an exponential polynomial over $F$, not equal to $\exp(g)$ for any exponential polynomial $g$, then there are $a_1,\ldots,a_n \in F$ such that $f(a_1,\ldots,a_n) = 0$.
\end{defn}
This statement was conjectured by Schanuel to hold in $\Cexp$, and Henson and Rubel \cite[Theorem~5.4]{HR84} proved that it does indeed hold there.

To show that a pure field is algebraically closed it is enough to know that every non-trivial polynomial has a root. The Schanuel nullstellensatz is an analogue of that statement, but it does not characterize strongly exponentially-algebraically closed exponential fields. 
\begin{theorem}
There are ELA-fields satisfying the Schanuel nullstellensatz which are not \seac.
\end{theorem}
\begin{proof}
Suppose that $F$ is an ELA-field and $F \into F'$ is a partial E-field extension generated by a solution $a_1,\ldots,a_n$ to an exponential polynomial $f$ (allowing iterations of exponentiation), not of the form $\exp(g)$. Following Shkop, $F'$ is also generated over $F$ by a tuple $\bbar$ such that $V_f \leteq \loc(\bbar,e^\bbar/F) \subs G^m$ is rotund, additively and multiplicatively free, and of dimension $m+n-1$. In particular, the extension is strong. The method is to add extra variables to remove the iterations of exponentiation, and then to remove variables to ensure freeness. It follows that some choice of $n-1$ of the $a_i$ are exponentially-algebraically independent over $F$, and the remaining one (say $a_1$) satisfies an exponential polynomial equation in one variable over $F \cup \{a_2,\ldots,a_n\}$. Thus if $F$ has infinite exponential transcendence degree, then it satisfies the Schanuel nullstellensatz iff it satisfies the same statement just for exponential polynomials in one variable.

Now if $F \strong F'$ is an E-field extension given by adjoining a root $a$ of an exponential polynomial $f$ in one variable, then $F' = \gen{F,a}^E_{F'}$. That is, as an E-field extension it is generated by the single element $a$. 

Define a perfectly rotund variety $V$ to have \emph{depth 1} iff in an extension $F \strong F|V$ with generating tuple $\abar$, there is a single element $c$ such that $\abar$ is contained in $\gen{F,c}^E$. Equivalently, $F \strong \gen{F,c}^E \strong F|V$. Since $\gen{F,c}^E$ is an E-field but not an ELA-field, such an intermediate field is possible.

Let $\Cat_1^{<\aleph_0}$ be the smallest category of finitely generated strong ELA-extensions $F$ of $SK^{ELA}$ which is closed under simple extensions which are either exponentially transcendental or given by perfectly rotund varieties of depth 1. Let $\Cat_1$ be the closure of $\Cat_1^{<\aleph_0}$ under unions of $\omega$-chains. Then, as in the proof of Proposition~\ref{etd determines extension}, $\Cat_1$ is an amalgamation category and hence has a unique Fra\"iss\'e limit, say $\mathcal{U}$. Then $\mathcal{U}$ satisfies the Schanuel nullstellensatz. However, there are perfectly rotund varieties $V$ which do not have depth 1 such as the generic hypersurface in $G^3$ used in the proof of Proposition~\ref{non-model-complete}. By Theorem~\ref{Jordan-Holder}, for such $V$ there is no $(\abar,e^\abar)$ in $\mathcal{U}$ which is generic in $V$ over a field of definition of $V$, and hence $\mathcal{U}$ is not \seac. 
\end{proof}

\section{Transcendence problems}

Schanuel's conjecture has many consequences in transcendence theory. Ribenboim \cite[pp323--326]{Ribenboim00} gives a few examples of easy consequences, one being that the numbers $e$, $\pi$, $e^\pi$, $\log \pi$, $e^e$, $\pi^e$, $\pi^\pi$, $\log 2$, $2^\pi$, $2^e$, $2^i$, $e^i$, $\pi^i$, $\log 3$, $\log \log 2$, $(\log 2)^{\log 3}$, and $2^{\sqrt{2}}$ are all algebraically independent. 
 
When Lang first published Schanuel's conjecture \cite[p31]{Lang66}, he wrote: 
\begin{quote}
``From this statement, one would obtain most statements about algebraic independence of values of $e^t$ and $\log t$ which one feels to be true.''
\end{quote}
We strengthen this empirical observation, and make it precise. To make a precise statement we need a precise definition.
\begin{defn}
A \emph{particular transcendence problem} is a problem of the following form:\\
Given complex numbers $a_1,\ldots,a_n$ with some explicit construction, what is the transcendence degree of the subfield $\Q(a_1,\ldots,a_n)$ of $\C$?
\end{defn}
The example above of Ribenboim is clearly of this form. All the numbers there are explicitly constructed from the rationals $\Q$ by the operations of exponentiation, taking logarithms, taking roots of polynomials, and field operations.

Let $\C_0 = \ecl^{\Cexp}(\emptyset)$ be the field of exponentially-algebraic complex numbers. By Fact~\ref{etd fact}, for any $\abar = (a_1,\ldots,a_n) \in \C^n$ such that no $\Q$-linear combination of them lies in $\C_0$, we have $\td(\abar,e^\abar/\C_0) \ge n+1$. Thus Schanuel's conjecture (for \Cexp) is equivalent to its restriction to $\C_0$. Recall that SK embeds in $\C_0$, and so Schanuel's conjecture for $\C_0$ is equivalent to the assertion that $SK \strong \C_0$. By Proposition~\ref{etd determines extension}, if $SK \strong \C_0$ then $B_0 \iso \EAC{\C_0}$ (recall $B_0 = \EAC{SK}$). Thus Schanuel's conjecture is equivalent to the assertion that $\C_0$ embeds in $B_0$. If in addition $\C_0$ were \seac, that is, $\C_0 = \EAC{\C_0}$, then there would be an isomorphism $\C_0 \iso B_0$. Since the automorphism group of $B_0$ is very large, such an isomorphism would be very far from being unique.

\begin{theorem}\label{td problems}
Schanuel's conjecture decides all particular transcendence problems where the complex numbers $a_1,\ldots,a_n \in \C$ are given by an explicit construction from $\Q$ by the operations of exponentiation, taking logarithms, taking roots of polynomials, field operations, and taking implicit solutions of systems of exponential polynomial equations.
\end{theorem}
\begin{proof}
The conditions on the $a_i$ are equivalent to them all lying in $\C_0$, that is, being exponentially-algebraic complex numbers. Assuming Schanuel's conjecture, $\C_0$ embeds in $B_0$. Any explicit description of the $a_i$ defines a finitely generated partial E-subfield $F$ of $B_0$, the smallest one containing all the coefficients of the exponential polynomial equations used in the given descriptions of the $a_i$. $F$ is necessarily strong in $B_0$, since it contains witnesses of all of its elements being exponentially algebraic. When taking logarithms or, more generally, taking implicit solutions of systems of equations, there are countably many solutions in $B_0$, but the homogeneity of $B_0$ for strong partial E-subfields (which follows from the Fra\"iss\'e theorem used in the proof of Proposition~\ref{etd determines extension}) shows that these choices do not affect the isomorphism type of $F$. Thus Schanuel's conjecture determines the isomorphism type of $F$ as a partial E-field, and hence it determines the transcendence degree of its subfield $\Q(a_1,\ldots,a_n)$.
\end{proof}

Note that if we do not allow taking implicit solutions of systems of exponential polynomial equations then the construction stays inside the field $SK^{ELA}$, and the proof depends only on section~2 of this paper. In particular this covers the field $SK^{EL}$ which, under Schanuel's conjecture, is the field of all of what Chow \cite{Chow99} calls \emph{EL-numbers}, that is, those complex numbers which have a closed-form representation using $0$, $+$, $\cdot$, $-$, $\div$, $\exp$ and the principal branch of the logarithm.

The construction will produce a generating set $\bbar$ for $D(F)$, and polynomial equations with rational coefficients determining the locus $V$ of $(\bbar,\exp(\bbar))$. If we had an algorithm to determine the $\Q$-linear relations holding on $\bbar$ and the multiplicative relations holding on $\exp(\bbar)$, that would give an algorithm for answering particular transcendence problems of this form.

There are other transcendence problems which are more general in nature, for example the four exponentials conjecture which states that if $x_1, x_2, y_1, y_2 \in \C$ and $\ldim_\Q(x_1,x_2) = \ldim_\Q(y_1,y_2) = 2$, then 
\[\td(e^{x_1y_1}, e^{x_1y_2},e^{x_2y_1},e^{x_2y_2}) \ge 1.\]
 The four exponentials conjecture is not a particular transcendence problem as defined above, but nonetheless it can easily be seen to follow from Schanuel's conjecture. So the statement of Theorem~\ref{td problems} is not a complete answer to formalising Lang's observation. Nonetheless, the method of proof above does apply. The four exponentials conjecture can be viewed as the conjunction of a set of particular transcendence problems, namely every specific instance of the problem. More generally, suppose $\mathcal{P}$ is a transcendence problem, such as the four exponentials conjecture, which asserts that some transcendence degree is large given suitable conditions (about exponentials, logarithms, and algebraic equations). Then either (every instance of) $\mathcal{P}$ is true in $B_0$ so it follows from Schanuel's conjecture that it is true in $\C_0$, or $\mathcal{P}$ is false in $B_0$, in which case, since $B_0$ is constructed in as free a way as possible, $\mathcal{P}$ cannot be true in any exponential field $F$ (unless it is true trivially because the hypotheses are not satisfied by any numbers in $F$).

\subsection*{Connection with conjectures on periods}

The two main conjectures about $\Cexp$ are:
\begin{enumerate}
\item Schanuel's conjecture, equivalently $\Cexp$ embeds in $\B$, equivalently $\C_0$ embeds in $B_0$;
\item $\Cexp$ is \seac, equivalently that $\Cexp = \EAC{\Cexp}$
\end{enumerate}
Together, they form Zilber's conjecture that $\Cexp \iso \B$. As shown in \cite{EAEF}, using work of Ax \cite{Ax71}, Schanuel's conjecture is equivalent to its restriction to $\C_0$. In the light of Lemma~\ref{hull lemma} and Proposition~\ref{almost strong}, the restriction of the conjecture to $\C_0$ is equivalent to the assertion that if $a$ is an exponentially algebraic complex number then there is a unique reason for that, meaning a unique smallest finite-dimensional $\Q$-vector subspace $\hull{a}$ of $\C$ containing $a$ such that $\delta(\hull{a}) = 0$.

This formulation of Schanuel's conjecture makes a visible connection with the conjecture of Kontsevich and Zagier on periods \cite[\S1.2]{KZ_Periods}. They conjecture that if a complex number is a period then there is a unique reason for that, up to three rules for for manipulating integrals: additivity, change of variables, and Stokes' formula. Kontsevich and Zagier give an alternative formulation of their conjecture in \cite[\S4.1]{KZ_Periods}. There is a canonical surjective homomorphism from a formal object, the vector space of effective periods, to the space of complex periods. The periods conjecture is equivalent to this homomorphism being an isomorphism. In the exponential case, the existence of automorphisms of $B_0$ means there can be no canonical isomorphism from the formal object $B_0$ to $\C_0$. Furthermore, since the objects in question are fields rather than vector spaces, there cannot be a non-injective map between them so if the conjecture is false then there is no map at all from $B_0$ to $\C_0$, although one could repair this by taking suitable subrings of $B_0$ instead. Finally, the open question of strong exponential-algebraic closedness of $\C_0$ means that any map should go from $\C_0$ to $B_0$ rather than the other way round, or that the subrings of $\B$ chosen should be restricted in some way. The power of the predimension method, as used in this paper, is that such considerations are not necessary.

The Kontsevich-Zagier conjecture does not imply Schanuel's conjecture, because for example $e$ is (conjecturally) not a period. Even the expanded conjecture on exponential periods \cite[\S4.3]{KZ_Periods} does not say much about Schanuel's conjecture, because (again conjecturally) $e^e$ is not an exponential period. Furthermore Schanuel's conjecture does not just refer to $\C_0$ but to all of $\C$ whereas periods form a countable subset of $\C$. Andr\'e has observed \cite[\S4.4]{Andre09} that the Kontsevich-Zagier conjecture is equivalent to Grothendieck's conjecture on periods, and Andr\'e himself proposed a conjecture which encompasses both Grothendieck's periods conjecture and Schanuel's conjecture \cite[\S5.8.1]{Andre09}, and applies to all of $\C$.

\section{Open Problems}

We end with some open problems. Schanuel's conjecture is known to be very difficult, and the conjecture that $\Cexp$ is \seac\ is also widely open (even assuming Schanuel's conjecture). We suggest some questions about complex exponentiation which may be easier.

\begin{enumerate}
 \item[(1)] Define an ELA-field $F$ to be \emph{locally finitely presented} iff
  every finitely generated ELA-subfield of $F$ is finitely presented. Is $\Cexp$ locally finitely presented?
 \item[(2)] Is there \emph{any} finitely presented exponential subfield of $\Cexp$?
 \item[(3)] Is there an exponential subfield $C$ of $\C$, and a finitely presented proper extension of $C$ realised inside $\C_0$, the subfield of exponentially algebraic numbers in $\C$? Since $\C_0 \strong \Cexp$, the question is resolved outside $\C_0$.
 \item[(4)] Let $V \subs G^n(\C)$ be perfectly rotund. The theorem of Henson and Rubel \cite[Theorem~5.4]{HR84} implies that if $n=1$ then there is $(a,e^a) \in V$ in \Cexp. How about $n=2$, or $n=3$? Indeed, for which $V$ can one show there are any solutions in $\Cexp$? 
 \item[(5)] Is there any perfectly rotund $V$ which is not of depth 1 with $(a,e^a) \in V$ in \Cexp?
\end{enumerate}

An apparently difficult problem is to construct an ordered analogue of pseudo-exponentiation which should be conjecturally elementarily equivalent to the real exponential field $\Rexp$. Since the real exponential function is determined just by its being a homomorphism which is order-preserving, continuous, and by the cut in the reals of $e$, one would have to assume Schanuel's conjecture for $\Rexp$ to construct an Archimedean model. The following problem is of the same nature, but may perhaps be more straightforward.
\begin{enumerate}
  \item[(6)] Can the automorphism $\sigma_0$ on $SK$ be extended to an automorphism of order 2 on a subfield of $B_{\aleph_0}$ larger than $SK^E$, such as $SK^{EA}$,  $SK^{ELA}$, $B_0$, or even $B_{\aleph_0}$ itself, in such a way that the exponential map is order-preserving on the fixed field (which will necessarily be real-closed, and hence ordered)?
\end{enumerate}

Finally, the predimension method used in this paper is very powerful, and can be extended beyond the exponential setting, for example to the exponential maps of semi-abelian varieties \cite{TEDESV} and to sufficiently generic holomorphic functions known as Liouville functions \cite{Zilber02tgfd}, \cite{Wilkie05}, and \cite{Koiran03}. The periods conjecture of Andr\'e encompasses the first of these settings and also the Grothendieck-Kontsevich-Zagier periods conjecture.

\begin{enumerate}
  \item[(7)] Is there a way to formulate Andr\'e's conjecture as the non-negativity of some predimension function, satisfying the essential properties such as the addition formula and submodularity?
\end{enumerate}

\end{document}